\documentclass{amsart}

\usepackage{amsmath, graphicx}
\usepackage{amsfonts}
\usepackage{amssymb}
\usepackage
{hyperref}
\hypersetup{colorlinks=true,citecolor=blue,linkcolor=blue,urlcolor=blue,
pdfstartview=FitH}

\newcommand{\sym}{\mathrm{sym}}

\newcommand{\intR}{\int_{-\infty}^{\infty}}
\newcommand{\half}{\ensuremath{ \frac{1}{2}}}
\newcommand{\sumeven}{\sideset{}{^\text{ev}}\sum}
\newcommand{\leg}[2]{\left(\frac{#1}{#2}\right)}

\numberwithin{equation}{section}

\newtheorem{theorem}{Theorem}[section]
\newtheorem{proposition}[theorem]{Proposition}
\newtheorem{lemma}[theorem]{Lemma}
\newtheorem{corollary}[theorem]{Corollary}
\newtheorem{conjecture}[theorem]{Conjecture}

\begin{document}

\title{Distribution of mass  of   holomorphic cusp forms}
\author{Valentin Blomer,  Rizwanur Khan and Matthew Young}
\address{Mathematisches Institut, Georg-August Universit\"{a}t G\"{o}ttingen, Bunsenstra\ss e 3-5, D-37073 G\"{o}ttingen, Germany}
\email{blomer@uni-math.gwdg.de, rrkhan@uni-math.gwdg.de}
\address{Department of Mathematics, Texas A\&{}M University, College Station, TX 77843-3368, U.S.A.}
\email{myoung@math.tamu.edu}
\thanks{V.B. supported by the Volkswagen Foundation. V.B. and R.K.  supported   by a Starting Grant of the European Research Council. M.Y. supported by the National Science Foundation under agreement No. DMS-0758235. Any opinions, findings and conclusions or recommendations expressed in this material are those of the authors and do not necessarily reflect the views of the National Science Foundation.}



\begin{abstract} We prove an upper bound for the $L^4$-norm and for the $L^2$-norm restricted to the vertical geodesic of a holomorphic Hecke cusp form $f$ of large weight $k$.  The method is based on Watson's formula and estimating a mean value of certain $L$-functions of degree 6.  Further applications to restriction problems of Siegel modular forms and subconvexity bounds of degree 8 $L$-functions are given.  
\end{abstract}

\maketitle

\section{Introduction} 

Suppose $f \in S_k$ is an $L^2$-normalized cuspidal Hecke eigenform of even weight $k$ for the modular group $\Gamma = SL_2(\Bbb{Z})$.  A basic question is to understand the size of $f$ and the distribution of its mass as $k$ becomes large; more precisely, we consider   $F(z) = y^{k/2} f(z)$ since $|F(z)|$ is $\Gamma$-invariant.  This can be made quantitative in various ways, e.g. by bounding the $L^p$-norm of $F$ for $2 < p \leq \infty$.   
A first guess might be that the mass of $F$ should be nicely distributed on $\Gamma\backslash \Bbb{H}$ such that $F$ has no essential peaks. Indeed, the mass equidistribution distribution conjecture, proved in \cite{HSo}, tells us that  the measure $|F(z)|^2 dx\, dy/y^2$  tends to the uniform measure $(3/\pi) \,dx\, dy/y^2$ (in the sense of integration against continuous and compactly supported test functions) as $F$ runs through a sequence of cuspidal Hecke eigenforms with weight $k$ tending to infinity.  A closer look, however, reveals that $F$ takes large values high in the cusp at $y = k/(4\pi)$, and for $p = \infty$ we have the essentially best-possible result
\begin{equation}
\label{Fsupbound}
   \|F \|_{\infty} = k^{\frac{1}{4}+o(1)},
\end{equation}   
see \cite{X}, which uses Deligne's bound.  A variant of this argument shows
\begin{equation}\label{lower}
  \| F \|_p \gg k^{\frac{1}{4} - \frac{3}{2p} - \varepsilon}
\end{equation}
(which is non-trivial only for $p > 6$), and in the opposite direction we have the interpolation (convexity) bound
\begin{equation}\label{triv}
  \| F \|_p \leq \| F \|_2^{2/p} \| F \|_{\infty}^{1-\frac{2}{p}} \ll k^{\frac{1}{4} - \frac{1}{2p}+\varepsilon}.
\end{equation}
We will give a quick proof of \eqref{lower} in Section \ref{sec3}. In this article we are interested in the $L^4$-norm of $F$ and its connection to $L$-functions.  In this case \eqref{triv} becomes $\| F \|_4^4 \ll k^{1/2 + \varepsilon}$, and nothing better has been known so far. Our first result constitutes an improvement over this trivial bound. 

\begin{theorem}\label{thm1} We have
\begin{displaymath}
 \| F \|_4^4 \ll k^{1/3+\varepsilon}. 
\end{displaymath}
\end{theorem}

Theorem \ref{thm1} shows that the measure of the set where $F$ satisfies \eqref{Fsupbound} is small.  One also immediately obtains an improvement on \eqref{triv} for all $2 < p < \infty$ by interpolation, namely,
\begin{equation*}
\| F \|_p \ll 
\begin{cases}
k^{\frac{2}{3}(\frac14 - \frac{1}{2p}) + \varepsilon}, \quad &\text {if } 2 \leq p \leq 4, \\
k^{\frac14 - \frac{2}{3 p} + \varepsilon}, \quad &\text {if } 4 \leq p < \infty.
\end{cases}
\end{equation*}

One may speculate on what is the true size of the $L^4$-norm.
\begin{conjecture}\label{thm2} On the basis of the conjectures in \cite{CFKRS}, one has with the normalization
\begin{equation}\label{norm}
  \int_{\Gamma \backslash \Bbb{H}} |f(z)|^2 y^k \frac{3}{\pi} \frac{dx \, dy}{y^2}  = 1,
\end{equation}
that as $k \rightarrow \infty$,
\begin{equation}\label{conj}
  \int_{\Gamma \backslash \Bbb{H}} |f(z)|^4 y^{2k} \frac{3}{\pi} \frac{dx \, dy}{y^2} = 2 + o(1).
 \end{equation}
\end{conjecture}
Note that with the normalization \eqref{norm}, Cauchy-Schwarz implies $\| F \|_4 \geq 1$. 

We remark on the side that for an $\infty$-old form $F$ of weight $k$, i.e.\  the ($L^2$-normalized) iterated Maa{\ss} lift $K_{k-2} \cdots K_2K_0 f$ of a fixed weight 0 cusp form $f$, Bernstein and Reznikov \cite[Section 2.6]{BR} have shown the unconditional bound $\| F \|_4 = O(1)$ of almost the same strength as \eqref{conj}, at least for (fixed) co-compact lattices, and together with \cite[Theorem A]{R}  the same bound should hold for (fixed) congruence subgroups\footnote{We would like to thank G. Harcos for pointing this out.}. \\

At first sight, the numerical value in \eqref{conj} is surprising in light of the following variation.
\begin{conjecture}\label{thm2'} Suppose $\phi$ is a Hecke-Maa{\ss} form for the full modular group with spectral parameter $T$.  On the basis of the conjectures in \cite{CFKRS}, one has with the normalization
\begin{displaymath}
  \int_{\Gamma \backslash \Bbb{H}} |\phi(z)|^2 \frac{3}{\pi} \frac{dx \, dy}{y^2}  = 1,
\end{displaymath}
that as $T \rightarrow \infty$,
\begin{equation}
\label{L4Maass}
  \int_{\Gamma \backslash \Bbb{H}} |\phi(z)|^4 \frac{3}{\pi} \frac{dx \, dy}{y^2} = 3 + o(1).
\end{equation}
\end{conjecture}
Conjecture \ref{thm2'} has been folklore for a while, see e.g.\ \cite[p.\ 989]{KR} and the discussion in \cite[\S 4]{Sa}.  Since  the fourth moment of a normalized Gaussian random variable is 3, it is  consistent with the random wave model of M.\ Berry \cite{Be}, and some numerical evidence is given, for instance, in \cite{HR, HSt}. Based on the usual analogy between large weight holomorphic cusp forms and Maa{\ss} forms, one might have expected the answer of $3$ in both conjectures, but as P.\ Sarnak pointed out to us, Conjecture \ref{thm2} indicates that $f(z) y^{k/2}$ is modelled by a \emph{complex} Gaussian for which the normalized fourth moment is 2.  One should keep in mind, however, that by \eqref{lower} this analogy ends certainly with the eighth moment which is not bounded any more.

Although \eqref{conj} and \eqref{L4Maass} look very pleasant using   probability measure, we nevertheless follow the usual convention in the literature and use $\frac{dx dy}{y^2}$ since this aids us in quoting results.  \\

One may ask the question of bounding $L^4$-norms in terms of other parameters of automorphic forms. Sarnak and Watson  \cite[Theorem 3] {Sa} can show the bound
  $\| f \|_4 \ll \lambda^{\varepsilon}$ 
for a weight $0$ Hecke-Maa{\ss} cusp form of large eigenvalue $\lambda$, possibly assuming the Ramanujan-Petersson conjecture (see also \cite{Lu2}).  For Eisenstein series restricted to fixed compact regions within $\Gamma \backslash \Bbb{H}$ this has been shown by Spinu \cite{Sp}. In the level aspect, a best-possible result on average has been proved in \cite{Bl}. All these results have Watson's formula \cite{Wa2} as a starting point that  translates the $L^4$-norm into a mean value of certain triple product $L$-functions of degree 8, but  they are of very different levels of difficulty. The present case of the weight aspect is the hardest in terms of the size of the conductors of the relevant $L$-functions. Here Watson's formula gives roughly
\begin{equation}\label{wat}
  \| F \|_4^4 \approx  \frac{1}{k} \sum_{g \in B_{2k} }L(1/2, f \times f \times g)
\end{equation}
where here and henceforth $B_k$ denotes a  Hecke basis of $S_k$. 
This is a family of about $k$ $L$-functions having conductors of size about $k^6$. The Lindel\"of hypothesis would imply $\| F \|_4^4 \ll k^{\varepsilon}$, but unconditionally   a bound of this strength seems to be completely out of reach by present technology.  Using the factorization 
\begin{equation}
\label{factorization}
  L(1/2, f \times f \times g) = L(1/2, \mathrm{sym}^2f \times g) L(1/2, g)
\end{equation}
and non-negativity of central $L$-values \cite{KZ, La}, one can estimate the second factor individually by $k^{1/3+\varepsilon}$, the best known subconvexity bounds for this degree 2 $L$-function \cite{Pe}, 
and is left with an average of degree 6 $L$-functions of conductor $k^4$ in a family of size $k$. Here we are in a position to obtain   a best-possible upper bound (``Lindel\"of on average") which is of independent interest. The following result is slightly more general than needed for our applications. 

\begin{theorem}\label{prop31} Fix a constant $c > 0 $.  For $f \in B_k$ and $|\kappa - k| \leq c$ we have  \begin{displaymath}
  \frac{12}{2k-1} \sum_{g \in B_{2\kappa}}    \frac{    L(1/2, \mathrm{sym}^2f  \times g)  }{  L(1, \mathrm{sym}^2 g)} \ll  k^{\varepsilon}.
\end{displaymath}
The implicit constant depends only on $\varepsilon$ and $c$. 
\end{theorem}

This is  the main ``workhorse'' result of the paper that is used in the course of proving Theorems \ref{thm1}, \ref{thmSK}, and \ref{cor2}. We will only need the cases $\kappa = k$ even  and $\kappa = k-1$ odd which come up naturally in our period formulae \eqref{L4normformula} and \eqref{ichino} below, but the argument works in greater generality as long as $k$ and $\kappa$ are sufficiently close (see below for a more detailed discussion). We note that Theorem \ref{prop31} is trivial in the case
$\kappa \geq k$, $\kappa$ odd, and in the case $k < \kappa$, $\kappa$ even, since in these cases  the root number of $L(s, \mathrm{sym}^2f \times g)$ is $-1$. \\

The factorization \eqref{factorization} together with a subconvexity bound for $L(1/2, g)$ gives trivially a subconvexity bound for the degree 8 function on the left hand side of \eqref{factorization}. Based on Theorem \ref{prop31} we can  get a subconvexity bound for a degree 8 $L$-function in a much less obvious situation. This seems to be the first instance of subconvexity for a triple product $L$-function with \emph{three varying} factors.

\begin{corollary}\label{subconvex} Let $k, l$ be two even positive integers and let $f \in B_k$, $h \in B_l$ be two Hecke eigenforms. Then
\begin{displaymath}
 \frac{12}{2(k+l)-1}\sum_{g \in B_{k+l}} \frac{L(1/2, f \times g \times h)}{L(1, \mathrm{sym}^2 g)} \ll  (kl)^{1/6+\varepsilon}.
\end{displaymath}
In particular, $L(1/2, f \times g \times h) \ll (k+l)(kl)^{1/6+\varepsilon}$ for each $g \in B_{k+l}$. 
\end{corollary}
The convexity bound in this situation is $((k+l)kl)^{1/2}$, so Corollary \ref{subconvex} gives  subconvexity in the range $k^{1/2+\delta} \leq l \leq k^{2-\delta}$. \\

The proof of Theorem \ref{prop31} is based on a careful study of the integral kernel in the $GL(3)$ Voronoi summation formula. It turns out that we roughly need to sum  
\begin{equation}\label{expr}
   \sum_{n \asymp  k^2} \lambda_f(n^2) J_{2\kappa-1}(\sqrt{n})
\end{equation}
where here and henceforth $\lambda_f$ denote the Hecke eigenvalues of $f$.  The Bessel function comes from Petersson's formula applied to the sum over $g \in B_{2\kappa}$. The key observation is that  large parts of the Voronoi kernel  are essentially cancelled by the Mellin transform of the Bessel function, and hence the seemingly complicated expression \eqref{expr}  with the Bessel function in the transitional region becomes treatable, cf.\ Lemma \ref{lem32}.  It is at this point that we need $k \approx \kappa $ in Theorem \ref{prop31}. A somewhat similar phenomenon was (implicitly) the key of success in X.\ ~Li's work \cite{Li}. The endgame of the proof features a stationary phase argument. For the purpose of this paper we could get by with an ad hoc argument, but a uniform analysis of oscillating integrals is a recurring theme in analytic number theory, and we felt that a general result in this direction may be welcome in many other situations. We give a weighted stationary phase lemma in Proposition \ref{statphase} below. It gives an asymptotic expansion with arbitrary precision, and it is also applicable in situations with several stationary points that move against each other, or in the case of mildly oscillating weight functions. \\

Theorem \ref{prop31} can be used in many situations, and we proceed to give two applications connected with norms of automorphic forms restricted to certain submanifolds.

For a holomorphic cuspidal Hecke eigenform $g \in S_{2k}$ with $k$ odd let $F_g \in S_{k+1}(\text{Sp}_4(\Bbb{Z}))$ be its Saito-Kurokawa lift (see \cite{EZ}).  Then $F_g$ restricted to the diagonal is a modular form on $(\Gamma\backslash \Bbb{H}) \times (\Gamma \backslash \Bbb{H})$, and we denote by $N(F_g)$ the (square of the) $L^2$-norm of this restricted function, when  both $\Gamma\backslash \Bbb{H}$ and $\text{Sp}_4(\Bbb{Z})\backslash \mathcal{H}_2$ are equipped with probability measures.  Ichino's formula \cite{Ic} implies  
\begin{equation}\label{ichino}
  N(F_g) = \frac{\pi^2}{15 \, L(3/2, g) L(1, \mathrm{sym}^2 g)}\cdot  \frac{12}{k} \sum_{f \in B_{k+1}} L\big(\tfrac{1}{2}, \mathrm{sym}^2 f \times g\big).
\end{equation}
It was conjectured in \cite{LY} that $N(F_g) \sim 2$ as $k \rightarrow \infty$, and  this conjecture was shown on average over both $g \in B_{2k}$ and $K \leq k \leq  2K$. Here we  show that the expected asymptotic formula holds for a much smaller average \emph{only} over $g \in B_{2k}$. 
\begin{theorem}\label{thmSK} We have
\begin{displaymath}
 \frac{12}{2k-1} \sum_{g \in B_{2k}} N(F_g) = 2 + O(k^{-\eta})
\end{displaymath}
for some $\eta > 0$. 
\end{theorem}
Dropping all but one term gives the bound $N(F_g) \ll k$ which is slightly  better  than the strongest individual bound obtained in \cite{LY}. With an amplifier one might even get a small power saving but we did not investigate this.  In this context Theorem \ref{prop31} has a geometric interpretation: the projection of the diagonally restricted $F_g$ onto any $f\times f$ with $f \in B_{k+1}$ is essentially bounded on average over  lifts $F_g$. It would be very interesting to prove a \emph{lower} bound in Theorem \ref{prop31} since this would show that for a given $f$, it is not the case that the projection of $F_g$ onto $f \times f$ is zero for all $g$.\\

As another application we let $f \in B_k$ be an $L^2$-normalized cuspidal Hecke eigenform and write as before $F(z) = f(z) y^{k/2}$. We consider the restriction of $F$ to the distinguished vertical \emph{infinite length} geodesic:
\begin{equation}\label{real}
  \mathcal{I} := \int_0^{\infty} |F(iy)|^2 \frac{dy}{y} = \int_0^{\infty} f(iy)^2 y^k \frac{dy}{y}.
\end{equation}
It follows easily from Parseval  that this   integral can be expressed in terms of $L$-functions (this is a classical observation of Hecke; a quick derivation is given in Section \ref{sec6}):
\begin{equation}\label{period}
\begin{split}
  \mathcal{I} =   \int_{-\infty}^{\infty} \frac{2^{k-2} |\Gamma(\frac{k}{2} + it)|^2} {\Gamma(k)} \cdot  \frac{|L(1/2 + it, f)|^2}{L(1, \mathrm{sym}^2 f)} dt  \sim \left(\frac{\pi}{2k}\right)^{1/2}  \int_{-\infty}^{\infty} e^{-2t^2/k} \frac{|L(1/2+it, f)|^2}{  L(1, \mathrm{sym}^2 f)} dt.
 \end{split} 
\end{equation}
One can show in various ways $\mathcal{I} \ll k^{1/2+\varepsilon}$, either by using \eqref{Fsupbound} or alternatively by  a mean value theorem for Dirichlet polynomials, while the Lindel\"of hypothesis would predict that $\mathcal{I} \ll k^{\varepsilon}$. 

The situation is once again in sharp contrast to the non-holomorphic case: the mean value theorem argument applied to 
  $\mathcal{J} := \int_0^{\infty} |\phi(iy)|^2 dy/y$ 
for an $L^2$-normalized Hecke-Maa{\ss} cusp form $\phi$ with large Laplace eigenvalue $1/4 + T^2$ shows immediately the essentially best-possible bound $\mathcal{J} \ll T^{\varepsilon}$, see \cite[p.6]{Sarnak}. 

We will conclude from Theorem \ref{prop31}   the following improvement on the trivial bound in the holomorphic case.
\begin{theorem}\label{cor2} We have
\begin{displaymath}
  \mathcal{I} \ll k^{1/4+\varepsilon}. 
\end{displaymath}
\end{theorem}
Theorem \ref{cor2}  shows that the measure of the set of $y > 0$ where $F(iy)$ satisfies \eqref{Fsupbound} is small.  Observe that the optimal bound $\mathcal{I} \ll k^{\varepsilon}$ would give, in light of \eqref{period}, an extremely strong subconvexity result, but even Theorem \ref{cor2} in its present form implies an interesting (``Burgess-type") hybrid subconvexity bound.  
Our approach to proving Theorem \ref{cor2} easily shows the (weaker) result that $\mathcal{I}^{1/2} \ll k^{1/8 + \varepsilon} \| F \|_4$ which is reminiscent of a result of Bourgain \cite{Bourgain} which compares the restricted $L^2$-norms along geodesics and the $L^4$-norm of Laplace eigenfunctions on a \emph{compact} Riemannian manifold. As a by-product of the calculations in   Section \ref{sec3} we will also show the  lower bound $\mathcal{I} \gg k^{-\varepsilon}$, see Corollary \ref{geodesic}.
 
As far as we know, this is the first nontrivial geodesic restriction result for holomorphic forms of large weight.  Reznikov \cite{Reznikov} initiated a study of restricted $L^2$-norms along various curves for Maa{\ss} forms with large eigenvalue.  Sarnak \cite{Sarnak} mentions that the restricted $L^2$-norm of $F(z)$ along a fixed closed horocycle is $O(k^{\varepsilon})$; this horocycle case amounts to bounding the sum of squares of Hecke eigenvalues of $f$ of size $\approx k$ but in a short interval of length $\sqrt{k}$.  This is very different from the analysis of $\mathcal{I}$.   Our approach to bounding $\mathcal{I}$ is specific to the vertical geodesic because we use the realness of $f$ on the geodesic in \eqref{real} which is exploited in \eqref{decomp}. \\

Returning to the situation of Theorem \ref{thm1}, we finally mention that rather than decomposing $f(z)^2$ into a Hecke basis of holomorphic forms, one could instead use a spectral decomposition for $y^k |f(z)|^2$.  In place of \eqref{wat}, we instead arrive at a mean-value of the shape
\begin{equation}\label{spectral}
\| F \|_4^4 \approx k^{-1} \sum_{t_j \ll \sqrt{k}} L(1/2, f \times f \times u_j) + k^{-1} \int_{t \ll \sqrt{k}} |L(1/2 + it, f \times f)|^2 dt,
\end{equation}
where only the even Maa{\ss} forms occur in the sum.
The conductor of this degree $8$ $L$-function is $t_j^4 k^4$ and it factorizes as $L(1/2, \mathrm{sym}^2 f \times u_j) L(1/2, u_j)$. In this case the degree $2$ factor has conductor $t_j^2 \ll k$ while the degree $6$ factor has conductor about $t_j^2 k^4 \ll k^5$; this has the effect that if one uses a subconvexity bound on the degree $2$ factor then one is left with estimating a family of about $ k$ $L$-functions having conductors of size about $k^5$, which is more difficult.  This alternate formulation also gives an independent way to derive Conjecture \ref{thm2}, and it does indeed lead to the same constant.  Since $f(iy)$ is real for $y > 0$, we may use  the decomposition of $f(z)^2$ into holomorphic forms also in the situation of Theorem \ref{cor2} which again works more efficiently than the corresponding decomposition of $y^k|f(z)|^2$ into Maa{\ss} forms.  On the other hand, the decomposition \eqref{spectral} can be used for the following variant of   Corollary \ref{subconvex}:  for $f$ and $g$ of weight $k$ and  $u_j$ an even Maa{\ss} form with spectral parameter $t_j$ one has  the bounds
\begin{equation}
\label{eq:fguj}
 L(1/2,f \times g \times u_j) \ll k^{4/3 + \varepsilon}, \qquad |L(1/2 + it, f \times g)|^2 \ll k^{4/3 + \varepsilon},
\end{equation}
provided $t_j, t \ll \sqrt{k}$. 
The conductors of these $L$-functions are $(k t_j)^4$ and $(kt)^2$, respectively, so these bounds are subconvex for $t_j \gg k^{1/3 + \delta}$ and $t \gg k^{1/3 + \delta}$, accordingly.
\\  

So far the results in this section have relied on the theory of $L$-functions.  It is also natural to attempt to bound these integrals directly with the Fourier expansion.  With this approach, we will show
\begin{theorem}
\label{L4Fourier} Let $f \in S_k$ be a Hecke eigenform and define $F(z) = f(z)y^{k/2}$ as before. 
Suppose $y_0 >0$.  Then
\begin{equation}
\label{SiegelIntegral}
   \int_{y_0}^{\infty} \int_{-1/2}^{1/2} |F(x+iy)|^4 \frac{dx dy}{y^2} \ll \frac{k^{1/2 + \varepsilon}}{y_0^2} + k^{-1/2 + \varepsilon}.
\end{equation}
\end{theorem}
In particular, this indicates that the bulk of the $L^4$-norm arises from small values of $y$, in contrast to \eqref{Fsupbound} where the supremum is attained very high in the cusp.  
The direct calculations with the Fourier expansions lead to sums of shifted convolution sums which when bounded trivially lead to Theorem \ref{L4Fourier}.  On the other hand, in certain ranges we can turn this analysis around and bound these new sums via Theorems \ref{thm1} and \ref{cor2}. We refer to Section \ref{sec3}, in particular Corollary \ref{cor32}, for the precise results on shifted convolution sums,   including a connection with Poincar\'e series.  \\  

Theorem \ref{L4Fourier} is somewhat reminiscent of \cite[Proposition 2]{So} which is a crucial input for quantum unique ergodicity for \emph{Maa{\ss} forms} on the modular surface; in essence it shows that mass (measured in the $L^2$-sense) cannot escape through the cusp. However, the methods in \cite{So}, based on the properties of multiplicative functions, are very different from ours.  \\

 
\textbf{Acknowledgements.} We would like to thank P.\ Sarnak and the referees for very useful comments. 

\section{Period and spectral formulae}\label{sec2}

In this section we compile several useful formulae for later use. In Subsection \ref{watsonformula} we can already deduce Theorem \ref{thm1} and Corollary \ref{subconvex} as well as the bounds \eqref{eq:fguj} from Theorem \ref{prop31} (whose proof is deferred to Section \ref{Sec4}). 
\subsection{The Petersson formula} Let
\begin{displaymath}
   E(z, s) = \sum_{\bar{\Gamma}_{\infty}\backslash \bar{\Gamma}} \Im(\gamma z)^s = y^s  + \frac{Z(2(1-s))}{Z(2s)} y^{1-s} + \frac{2\sqrt{y}}{Z(2s)} \sum_{n \not = 0}\frac{\tau_{s-\frac{1}{2}}(|n|)}{|n|^{1/2}}  K_{s-\frac{1}{2}}(2 \pi |n| y) e(nx)
\end{displaymath}  
denote the usual Eisenstein series where $Z(s) = \zeta(s) \Gamma(s/2) \pi^{-s/2}$ is the completed zeta-function, $\tau_{\nu}(n) = \sum_{ab = n} (a/b)^{\nu}$,   $\bar{\Gamma} = PSL_2(\Bbb{Z})$ and $\bar{\Gamma}_{\infty}$ is the  subgroup of upper triangular matrices in $\bar{\Gamma}$.  Let 
\begin{displaymath}
  g(z) = \sum_{n=1}^{\infty} \lambda_g(n) (4\pi n)^{(k-1)/2} e(nz) \in S_k
\end{displaymath}  
be a Hecke normalized cusp form, and write $G(z) = y^{k/2} g(z)$.  Then by unfolding
\begin{displaymath}
 \int_{\Gamma \backslash \Bbb{H}}  |G(z)|^2 E(z, s)  \frac{dx\, dy}{y^2} =  \int_0^{\infty}\sum_{n=1}^{\infty} |\lambda_g(n)|^2 (4\pi n)^{k-1} e^{-4\pi ny} y^{s+k} \frac{dy}{y^2} = \frac{L(s, g \times \bar{g})\Gamma(s+k-1)}{\zeta(2s) (4\pi)^s}
\end{displaymath}
in $\Re s > 1$. In particular,  
\begin{equation}\label{normholo}
  \| G \|_2^2 =  \frac{\pi}{3} \underset{s=1}{\text{res}} \frac{L(s, g \times \bar{g})\Gamma(s+k-1)}{\zeta(2s) (4\pi)^s} = \frac{L(1, \mathrm{sym}^2 g)\Gamma(k)}{12 \zeta(2)}. 
\end{equation}
 Combining this with the Petersson formula \cite[Proposition 14.5]{IK}, we obtain 
\begin{equation}\label{peter}
  \frac{\zeta(2)}{(k-1)/12} \sum_{g \in B_k} \frac{\lambda_g(n) \bar{\lambda}_g(m)}{L(1, \mathrm{sym}^2 g)} = \delta_{n, m} + 2\pi i^{-k} \sum_{c = 1}^{\infty} \frac{S(m, n, c)}{c} J_{k-1}\left(\frac{4\pi \sqrt{mn}}{c}\right)
  \end{equation}
 where we recall that $B_k$ denotes a Hecke  basis $B_k$ of $S_k$. 
 
 \subsection{The Voronoi formula}
 
  Let $\psi$ be a smooth function with compact support in $(0, \infty)$ with Mellin transform $\tilde{\psi}(s)$. Let $f \in S_k$ be a holomorphic Hecke cusp form of weight $k$ and denote by $A(n, m) = A(m, n)$ the Fourier-Whittaker coefficients of the symmetric square lift of $f$,  normalized such that $A(1,1) = 1$,  see \cite[Sections 6, 7]{Go}. Let $c$ be a natural number and $d$ an integer coprime to $c$. Then we have \cite[Theorem 1.18]{MS}
\begin{equation}\label{Vor}
  \sum_{n \geq 1} A(m, n) e\left(\frac{n\bar{d}}{c}\right) \psi(n) = c \sum_{\pm} \sum_{n_1 \mid c} \sum_{n_2 \geq 1} \frac{A(n_2, n_1)}{n_2n_1} S\left(md, \pm n_2, \frac{c}{n_1}\right)\Psi^{\pm}\left(\frac{n_2 n_1^2}{c^3m}\right)
\end{equation}
where
\begin{equation}\label{Psi}
  \Psi^{\pm}(x) = \frac{1}{2\pi^{3/2}} \int_{(1)} (\pi^3 x)^{-s} G^{\pm}(s) \tilde{\psi}(-s) \frac{ds}{2\pi i}
\end{equation}
with
 \begin{equation}\label{G}
  G^{\pm}(s) = \frac{\Gamma(\frac{1}{2}(k+1+s))\Gamma(\frac{1}{2}(k+s))}{\Gamma(\frac{1}{2}(k-s))\Gamma(\frac{1}{2}(k-1-s ))}  \left(\frac{\Gamma(\frac{1}{2}(2+s)}{\Gamma(\frac{1}{2}(1-s))} \mp i \frac{ \Gamma(\frac{1}{2}(1+s))}{\Gamma(\frac{1}{2}(-s))}\right). 
 \end{equation}

\subsection{Watson's formula}\label{watsonformula} Let $k, l$ be two even positive integers and let $f \in B_k$, $h \in B_l$, $g \in B_{k+l}$ be three Hecke eigenforms. We write $F = f y^{k/2}$, $H = h y^{l/2}$, $G = g y^{(k+l)/2}$. Then Watson's formula \cite[Theorem 3]{Wa2} together with the local computations in  \cite[Section 4.1]{Wa2} shows
\begin{equation}\label{watson}
\begin{split}
  |\langle FH, G \rangle|^2& = \frac{\Lambda(1/2, f \times \bar{g} \times h)}{4\Lambda(1, \mathrm{sym}^2 f)\Lambda(1, \mathrm{sym}^2 \bar{g}) \Lambda(1, \mathrm{sym}^2 h)} \\
  &= \frac{\pi^3}{2(k+l-1)} \cdot  \frac{L(1/2, f \times \bar{g} \times h)}{L(1, \mathrm{sym}^2 f)L(1, \mathrm{sym}^2 \bar{g})L(1, \mathrm{sym}^2 h)}.
  \end{split}
\end{equation}
Since $f, g, h$ have real Fourier coefficients, we can drop the complex conjugation bars. Applying this with $k=l$ and $f=h$, we obtain
\begin{equation}
\label{L4normformula}
  \| F \|_4^4 = \langle F^2, F^2 \rangle = \sum_{g \in B_{2k}} |\langle F^2, G\rangle|^2 =  \frac{\pi^3}{2(2k-1)L(1, \mathrm{sym}^2 f)^2}  \sum_{g \in B_{2k}} \frac{ L(1/2, g)  L(1/2, \mathrm{sym}^2f  \times g)  }{  L(1, \mathrm{sym}^2 g)}.
 \end{equation}
 On the other hand,  \eqref{watson} implies 
 \begin{equation}
\label{product}
 \begin{split}
   \langle |F|^2, |H|^2 \rangle &  = \langle FH, FH\rangle = \sum_{g \in B_{k+l}} |\langle FH, G\rangle|^2\\
   & = \frac{\pi^3}{2(k+l-1)L(1, \mathrm{sym}^2 f) L(1, \mathrm{sym}^2 h)}\sum_{g \in B_{k+l}}  \frac{L(1/2, f \times g \times h)}{ L(1, \mathrm{sym}^2 g) }.
   \end{split}
 \end{equation}

We see that Theorem \ref{thm1} is  an easy consequence of Theorem \ref{prop31}: in \eqref{L4normformula} we use the non-negativity of $L(1/2, \mathrm{sym}^2 f \times g)$ \cite{JS, La}  and  $L(1/2, g)$ \cite{KZ} together with the lower bound $L(1, \mathrm{sym}^2 f) \gg k^{-\varepsilon}$ \cite{HL}. Then Theorem \ref{thm1}  and  the individual  subconvexity  bound $L(1/2, g) \ll k^{1/3+\varepsilon}$   \cite[p.\ 37]{Pe}   imply Theorem \ref{thm1}.     \\

The same argument gives Corollary \ref{subconvex}: since $\langle |F|^2, |H|^2\rangle \leq \| F \|_4^2 \| H \|_4^2$,  Theorem \ref{thm1} implies Corollary \ref{subconvex}. \\

Finally we show how \eqref{eq:fguj} follows from the same ideas as above.  Suppose that $f$ and $g$ both have weight $k$.  For $u_j$ even, Watson's formula gives
\begin{equation*}
\begin{split}
 |\langle F\overline{G}, u_j \rangle|^2& =  \frac{\Lambda(1/2, f \times \bar{g} \times u_j)}{8\Lambda(1, \mathrm{sym}^2 f)\Lambda(1, \mathrm{sym}^2 g) \Lambda(1, \mathrm{sym}^2 u_j)}\\
 & =  2\cdot \frac{  L(1/2, f \times \bar{g} \times u_j)  }{L(1, \mathrm{sym}^2 f)L(1, \mathrm{sym}^2 g) L(1, \mathrm{sym}^2 u_j)} \mathcal{G}(k, t_j), \quad \mathcal{G}(k, t) := 
\frac{\pi^3    |\Gamma(k-\frac{1}{2} + it)|^2}{4\Gamma(k)^2}. 
\end{split} 
\end{equation*}
A straightforward computation with Stirling's formula shows
$\mathcal{G}(k, t) \sim \frac{\pi^3}{4}  k^{-1} \exp(-t^2/k)$ for $|t| \leq k^{2/3}$, and is exponentially small for $|t| > k^{2/3}$.  The classical Rankin-Selberg theory computes the projection of $F \overline{G}$ onto the Eisenstein series and the formula is
\begin{equation*}
  \frac{1}{4\pi}   |\langle F \overline{G}, E(., 1/2 + it) \rangle|^2  =   \frac{1}{\pi} \cdot\frac{|L(\frac{1}{2} + it, f \times \bar{g})|^2}{L(1, \mathrm{sym}^2 f)L(1, \mathrm{sym}^2 g)|\zeta(1+2it)|^2}  \mathcal{G}(k, t) .
\end{equation*}
As above we deduce $k^{1/3+\varepsilon} \gg \| F \|_4^2 \| G \|_4^2 \geq \langle F\bar{G}, F\bar{G} \rangle$ from Theorem \ref{thm1}; spectrally decomposing $F\bar{G}$ and using the preceding two inner product formulae   easily leads to \eqref{eq:fguj}.

\section{The Fourier expansion}\label{sec3}
In this section we sketch the proof of \eqref{lower} (which is a generalization of the method of \cite{X}), and prove Theorem \ref{L4Fourier}. 
Let 
\begin{equation}
\label{eq:fFourierExpansion}
  f(z) = \sum_{n=1}^{\infty} a_n (4\pi n)^{(k-1)/2} e(nz) \in S_k
\end{equation}
be an $L^2$-normalized holomorphic Hecke cusp form of weight $k$. Then  
\begin{equation*}
  |a_1|^2 \asymp \frac{1}{\Gamma(k) L(1, \mathrm{sym}^2f)} = \frac{ 1}{\Gamma(k/2)^2 k^{1/2+o(1)} 2^k},
\end{equation*}  
by \eqref{normholo}. It follows that
\begin{displaymath}
\begin{split}
 \| F \|_p & \geq \left(\int_{1}^{\infty} \int_0^1 \bigl|f(x+iy) y^{k/2}\bigr|^p \frac{dx \, dy}{y^2}\right)^{1/p} \geq \left(\int_{1}^{\infty} \Bigl| \int_0^1 f(x+iy) e(-x)  dx \Bigr|^p y^{kp/2}  \frac{dy}{y^2}\right)^{1/p}\\
 & = \left(\int_{1}^{\infty} \Bigl| a_1 (4\pi)^{(k-1)/2}e^{-2\pi y} y^{k/2}\Bigr|^p \frac{dy}{y^2}\right)^{1/p} \gg k^{-\varepsilon} \left(\int_{1}^{\infty} \Bigl| \frac{e^{-2\pi y} (2\pi y)^{k/2}}{\Gamma(k/2) k^{1/4}}\Bigr|^p \frac{dy}{y^2}\right)^{1/p}. 
\end{split} 
\end{displaymath}
Let $\mathcal{L} := [\frac{k}{4\pi} - \sqrt{k}, \frac{k}{4\pi} + \sqrt{k}]$. It is well-known that $e^{-2\pi y} (2\pi y)^{k/2} \gg \Gamma(k/2) k^{1/2}$ for $y \in \mathcal{L}$. Hence
\begin{displaymath}
  \| F \|_p \gg k^{-\varepsilon} \left(\int_{\mathcal{L} } k^{p/4} \frac{dy}{y^2} \right)^{1/p} \gg k^{\frac{1}{4} - \frac{3}{2p} - \varepsilon},
\end{displaymath}
as claimed. \\
 
Now we prove Theorem \ref{L4Fourier}.  Let  $P(y_0)$ denote the left hand side of \eqref{SiegelIntegral}. Writing out the Fourier expansion and integrating over $x$, we obtain
\begin{multline*}
 P(y_0) = |a_1|^4 \sum_{m + n = m' + n'} \lambda_f(m)  \lambda_f(n)  \lambda_f(m')  \lambda_f(n') (4 \pi m)^{\frac{k-1}{2}} (4 \pi n)^{\frac{k-1}{2}} (4 \pi m')^{\frac{k-1}{2}} (4 \pi n')^{\frac{k-1}{2}}
\\
\times \int_{y_0}^{\infty} y^{2k-1} \exp(-2 \pi y (m + n + m' + n')) \frac{dy}{y}.
\end{multline*}
Changing variables $y \rightarrow y/(2 \pi(m+n+m'+n'))$, we recast this as 
\begin{equation*}
   \frac{|a_1|^4}{2\pi} \sum_{m + n = m' + n'} \frac{\lambda_f(m)  \lambda_f(n)  \lambda_f(m')  \lambda_f(n')}{m+n+m'+n'} \Bigl(\frac{\sqrt{mn}}{m+n}\Bigr)^{k-1} 
\Bigl(\frac{\sqrt{m'n'}}{m'+n'}\Bigr)^{k-1}
\Gamma\bigl(2k-1, 2 \pi y_0(m+n+m'+n')\bigr),
\end{equation*}
where  
\begin{equation*}
 \Gamma(a,z) = \int_z^{\infty} t^{a} e^{-t} \frac{dt}{t}
\end{equation*}
is the incomplete gamma function.  Define $Q(a,x) =  \Gamma(a,x)/\Gamma(a)$ where $a,x>0$.   This function is well understood asymptotically. All we need here is that  $Q(a, x)$ is exponentially small for $x \geq a + \sqrt{a} \log{a}$, and $Q(a,x)-1$ is exponentially small for $x \leq a - \sqrt{a} \log{a}$; we always have $Q(a,x) \leq 1$.
 With \eqref{normholo} and letting $l = m+n = m'+n'$ be a new variable, we obtain 
 \begin{equation*}
P(y_0) =  \frac{2\pi^{5/2} \Gamma(k- \frac{1}{2})}{\Gamma(k) L(1, \mathrm{sym}^2 f)^2} \sum_{l} \frac{T_f(l)^2}{l} Q(2k-1, 4 \pi y_0 l)
\end{equation*}
where
\begin{equation*}
 T_f(l) = \sum_{m+n=l} \lambda_f(m) \lambda_f(n) \Bigl(\frac{2 \sqrt{mn}}{m+n} \Bigr)^{k-1}.
\end{equation*}
It turns out that $T_f(l)$ is closely related to the inner product of $f^2$ onto the $l$-th holomorphic Poincare series of weight $2k$; see \eqref{eq:fsquaredPl} below.
Unless $m \sim n$, the weight function in the definition of $T_f(l)$ is exponentially small.  Note that
\begin{equation*}
\frac{2\sqrt{mn}}{m+n} = 1 - \frac{|m-n|^2}{(m+n)(\sqrt{m} + \sqrt{n})^2}= 1 - \frac{|m-n|^2}{2(m+n)^2} + O\Bigl(\frac{|m-n|^4}{(m+n)^4}\Bigr),
\end{equation*}
so that the contribution to $T_f(l)$ from $|m-n| \geq \frac{l}{\sqrt{k}} \log(l)$ is exponentially small.  Then by Deligne's bound,
\begin{equation}
\label{Tbound}
 T_f(l) \ll l^{\varepsilon}\Bigl( 1+ \frac{l}{\sqrt{k}}\Bigr).
\end{equation}
At this point we can already estimate trivially to obtain $P(y_0) \leq k^{1/2+\varepsilon} y_0^{-2} + k^{-1/2+\varepsilon}$ as claimed in Theorem \ref{L4Fourier}.\\

For convenience we slightly simplify the expression for $P(y_0)$.  We first note the simple approximation
\begin{equation}\label{defSf}
 T_f(l) = S_f(l) +O(l^{1+\varepsilon} k^{-3/2}), \qquad S_f(l) =\sum_{m + n = l} \lambda_f(m) \lambda_f(n) \exp\Bigl(-\frac{|m-n|^2 k}{2 l^2}\Bigr).
\end{equation}
We may replace $Q(2k-1, 4 \pi y_0 l)$ by $1$ under the assumption $l \leq \frac{k}{2 \pi y_0}$, obtaining
\begin{equation}\label{triplesum}
P(y_0)=  \frac{2\pi^{5/2} \Gamma(k- \frac{1}{2})}{\Gamma(k) L(1, \mathrm{sym}^2 f)^2} \sum_{l \leq \frac{k}{2 \pi y_0}} \frac{ S_f(l)^2}{l} + O(k^{\varepsilon}y_0^{-2} + k^{-1 + \varepsilon}).
\end{equation}
 We deduce some additional corollaries from this argument. 
First we observe that the same argument can be used for the geodesic restriction problem in Theorem \ref{cor2} which we complement by the following result.

\begin{corollary}\label{geodesic} With the notation and assumptions of Theorem \ref{cor2}, we have 
\begin{equation}
\label{SiegelIntegral2}
R(y_0):= \int_{y_0}^{\infty} |F(iy)|^2 \frac{dy}{y} \ll \frac{k^{1/2 + \varepsilon}}{y_0} + k^{\varepsilon}.
\end{equation}
Furthermore, $R(1) \gg k^{-\varepsilon}$.
\end{corollary}
Indeed, a direct calculation shows
\begin{equation*}
R(y_0) := \int_{y_0}^{\infty} y^k |f(iy)|^2 \frac{dy}{y} = \frac{\pi^2 }{L(1, \mathrm{sym}^2 f)} \sum_{l} \frac{ T_f(l)}{l} Q(k, 2 \pi y_0 l)
\end{equation*}
with $T_f(l)$ as in \eqref{defSf}, and \eqref{Tbound} immediately implies the upper bound in \eqref{SiegelIntegral2}. 
With the same approximations as above, we obtain the slightly nicer expression
\begin{equation}\label{Ry0}
R(y_0) = \frac{\pi^2 }{L(1, \mathrm{sym}^2 f)} \sum_{l \leq \frac{k}{2 \pi y_0}} \frac{ S_f(l)}{l} + O\left(k^{\varepsilon} \Bigl(\frac{y_0}{k} + \frac{1}{y_0}\Bigr)\right).
\end{equation}
For a proof of the lower bound in Corollary \ref{geodesic} we observe that $R(1) \geq R(y_0)$ for $y_0 \geq 1$, and we choose $2 \pi y_0 = k^{1/2 + \varepsilon}$.  In this case, $l \leq k^{1/2-\varepsilon}$ and so effectively only the diagonal terms $m=n = l/2$, with $l$ even, persist in \eqref{defSf}.  That is,
\begin{equation*}
R(1) \geq \frac{\pi^2 }{L(1, \mathrm{sym}^2 f)} \sum_{2l \leq k^{1/2-\varepsilon}} \frac{\lambda_f(l)^2}{2l}  + O(k^{-1/2 + \varepsilon}).
\end{equation*}
Dropping all but $l=1$, we obtain the claimed lower bound.\\  

The expressions \eqref{triplesum} and \eqref{Ry0} can be used to bound on average the shifted convolutions $S_f(l)$ defined in \eqref{defSf}. 

\begin{corollary}\label{cor32}
Let $N \geq 1$. With the notation and assumptions as above  we have
\begin{equation}
\label{eq:ShiftedSum}
\sum_{l \leq N} \frac{S_f(l)}{l} \ll (Nk)^{\varepsilon}\left( k^{1/4} + \frac{N}{k}\right), \qquad \sum_{l \leq N} \frac{S_f(l)^2}{l} \ll (Nk)^{\varepsilon}\left(k^{5/6} + \frac{N}{k^{1/6}} +  \frac{N^2}{k^{3/2}} \right).
\end{equation}
\end{corollary}
The former bound is nontrivial for $N > k^{3/4 + \varepsilon}$, while the latter is nontrivial for $N > k^{11/12 + \varepsilon}$. This seems the first bound of this type in the literature.\\ 

To prove the first bound in \eqref{eq:ShiftedSum}, we apply \eqref{Ry0} with $2 \pi y_0 = k/N$ and  use the obvious inequality $R(y_0) \leq R(0) =\mathcal{I}$ in combination with Theorem \ref{cor2}.  
For  the second bound in \eqref{eq:ShiftedSum} we apply \eqref{triplesum}  with  $2 \pi y_0 = k/N$ and use the inequality $P(y_0) \ll (1+ y_0^{-1})P(1)$ (\cite[Lemma 2.10]{Iw0})  in combination with Theorem \ref{thm1}.  \\

The shifted convolution sum $T_f(l)$ is a natural object and can be interpreted in terms of Poincar\'e series as we now briefly explain.  Let $P_l$ denote the $l$-th holomorophic Poincar\'e series of weight $2k$ for the group $\Gamma = SL_2(\mathbb{Z})$ as in \cite[Section 3.3]{IwaniecTopics}, that is,
\begin{equation*}
 P_l(z) = \sum_{\gamma \in \bar{\Gamma}_{\infty} \backslash \bar{\Gamma}} j(\gamma, z)^{-2k} e(l \gamma z)
\end{equation*}
and define the normalized function $\widetilde{P}_l$ via
$$P_l(z) = \frac{\sqrt{\Gamma(2k-1)}}{(4 \pi l)^{(2k-1)/2}} \widetilde{P}_l(z).$$  This normalization is natural because by  \cite[(3.24)]{IwaniecTopics}, $\langle \widetilde{P}_l, \widetilde{P}_l \rangle$ is $1$ plus a sum of Kloosterman sums.
For a cusp form $g(z)$ of weight $2k$, we have
\begin{equation*}
 \langle g, \widetilde{P}_l \rangle = \frac{\sqrt{\Gamma(2k-1)}}{(4 \pi l)^{(2k-1)/2}} \widehat{g}(l),
\end{equation*}
where $g(z) = \sum_{l \geq 1} \widehat{g}(l) e(lz)$.  Suppose that $f$ of weight $k$ is given by \eqref{eq:fFourierExpansion}, and let $g = f^2$.  Since
\begin{equation*}
\widehat{g}(l) = a_1^2 \sum_{m+n=l} \lambda_f(m) \lambda_f(n) (4 \pi \sqrt{mn})^{k-1},
\end{equation*}
we obtain 
\begin{equation}
\label{eq:fsquaredPl}
T_f(l) =  \frac{2^{k-1} \sqrt{4\pi l}}{a_1^2 \sqrt{\Gamma(2k-1)}} \langle f^2, \widetilde{P}_l \rangle = k^{\frac{1}{4} + o(1)} l^{\frac{1}{2}} \cdot  \langle f^2, \widetilde{P}_l \rangle. 
\end{equation}

\section{Conditional results} 
Our next aim is  to   show how Conjecture \ref{thm2} follows from   the general recipe of \cite{CFKRS}.    The overall approach is analogous to the derivation of Conjecture 1.6 of \cite{LY} which is slightly different in that it averages  $L(1/2, \mathrm{sym}^2 f \times g)$ over $f$ while here we average $L(1/2, f \times f \times g) = L(1/2, \mathrm{sym}^2 f \times g) L(1/2, g)$ over $g$.  We assume some familiarity with \cite{CFKRS}. The forthcoming calculations are purely formal and only at the end do we arrive at something that makes sense. 
Mimicking the approximate functional equation we write formally
\begin{equation*}
L(1/2 + \alpha, g) = \sum_{l} \frac{\lambda_g(l)}{l^{1/2 + \alpha}} + X_{\alpha} \sum_{l} \frac{\lambda_g(l)}{l^{1/2 - \alpha}},
\end{equation*}
and
\begin{equation*}
L(1/2 + \beta, \mathrm{sym}^2 f \times g) = \sum_{m, n} \frac{\lambda_g(n) A(m,n)}{(m^2 n)^{1/2 + \beta}} + Y_{\beta} \sum_{m, n} \frac{\lambda_g(n) A(m,n)}{(m^2 n)^{1/2 - \beta}}
\end{equation*}
for certain quantities $X_{\alpha}$, $Y_{\beta}$ with $X_{0} = Y_{0} = 1$. As above, $A(m, n)$ denotes the Fourier-Whittaker coefficients of the symmetric square lift of $f$.  
Then by \eqref{L4normformula} we have
\begin{equation}
\label{CFKRS1st}
\| F \|_4^4  = \frac{\pi^3}{2(2k-1)L(1, \mathrm{sym}^2 f)^2} \sum_{g \in B_{2k}}   \sum_{l, m, n} \frac{\lambda_g(l) \lambda_g(n) A(m,n)}{l^{1/2 + \alpha} (m^2 n)^{1/2 + \beta}}+ \dots,
\end{equation}
where the dots indicate three more similar terms.  The Petersson formula \eqref{peter}  expresses this spectral sum as a diagonal term plus a sum of Kloosterman sums.
The \cite{CFKRS} conjecture instructs us to apply this averaging formula to each of the four terms in \eqref{CFKRS1st}, and to retain only the diagonal term.  Thus we obtain
\begin{equation*}
\| F \|_4^4   \sim   \frac{\pi^3}{24 \zeta(2) L(1, \mathrm{sym}^2 f)^2}   \sum_{m,n} \frac{A(m,n)}{m^{1 + 2 \beta} n^{1 + \alpha + \beta}} + \dots,
\end{equation*}
the dots indicating three similar terms obtained by switching the signs on the $\alpha$'s and $\beta$'s. It follows easily from the Hecke relations that 
 the Dirichlet series is
\begin{equation*}
\sum_{m,n} \frac{A(m,n)}{m^{1 + 2 \beta} n^{1 + \alpha + \beta}}  = \frac{L(\mathrm{sym}^2 f, 1 + 2\beta)L(\mathrm{sym}^2 f, 1 + \alpha + \beta)}{\zeta(2 + \alpha +  3 \beta)},
\end{equation*}
see \cite[Prop.\ 6.6.3]{Go}.  At this point we can set all the parameters to $0$, giving
\begin{equation*}
\|F \|_4^4   \sim \frac{\pi^3}{6\zeta(2)^2}  = \frac{6}{\pi}. 
\end{equation*}
Normalizing as in \eqref{norm}, we finally arrive at
 \begin{equation*}
   \int_{\Gamma \backslash \mathbb{H}} y^{2k} |f(z)|^4  \frac{3}{\pi} \frac{dx dy}{y^2}   \sim 2.
\end{equation*}

Next we indicate the changes necessary to derive Conjecture \ref{thm2'}.  Let $\phi$ be as in Conjecture \ref{thm2'}, and suppose $u_j$ form a Hecke-Maa{\ss} orthonormal basis for $SL_2(\mathbb{Z})$ with spectral parameter $t_j$.  Then the spectral decomposition gives
\begin{equation*}
\| \phi \|_4^4 = \int_{\Gamma \backslash \mathbb{H}} |\phi(z)|^4 \frac{dx dy}{y^2} = \sum_{j} |\langle \phi^2, u_j \rangle|^2 + (\text{Eisenstein}).
\end{equation*}
Watson's formula gives for $u_j$ even that
\begin{equation*}
 |\langle \phi^2, u_j \rangle|^2 = \frac{\pi}{2^3} H_T(t_j)   \frac{L(\phi \times \phi \times \overline{u_j}, 1/2)}{L(\mathrm{sym}^2 \phi, 1)^2 L(\mathrm{sym}^2 u_j, 1)}, \quad H_T(t) = \frac{|\Gamma(\frac{\half + 2i T + it}{2})|^2 |\Gamma(\frac{\half + 2i T - it}{2})|^2 |\Gamma(\frac{\half + it}{2})|^4}{|\Gamma(\frac{1 + 2iT}{2})|^4 |\Gamma(\frac{1 + 2it}{2})|^2}.
\end{equation*}
There is a similar formula for the projection of $\phi^2$ onto the Eisenstein series that follows much more elementarily from unfolding.  As above, we then obtain 
\begin{equation*}
 \| \phi \|_4^4 = \frac{3}{\pi} + \frac{\pi}{2^3 L(\mathrm{sym}^2 \phi, 1)^2} \sumeven_{j \geq 1} \frac{H_T(t_j)}{L(1, \mathrm{sym}^2 u_j)} L(1/2, u_j) L(1/2, \mathrm{sym}^2 \phi \times u_j) + (\text{Eis.}),
\end{equation*}
taking into account the constant eigenfunction $u_0 = \sqrt{3/\pi}$.  
Now the Kuznetsov formula plays the role of the Petersson formula.  To this end, we recall that the Kuznetsov formula takes the form
\begin{equation*}
 2 \sumeven_{j \geq 1} \frac{h(t_j) \lambda_j(m) \lambda_j(n)}{L(1, \mathrm{sym}^2 u_j)}  + (\text{Eis.}) = \frac12 \delta_{m=n} \intR h(t) d^*t + (\text{Kloosterman}), \qquad d^*t = \frac{1}{\pi^2} t \tanh(\pi t) dt.
\end{equation*}
We then arrive at the conjecture
\begin{equation*}
 \| \phi \|_4^4 - \frac{3}{\pi} \sim \frac{\pi}{8 \zeta(2)} I, \qquad I =\intR H_T(t) d^*t.
\end{equation*}
We next evaluate $I$.  
Stirling's formula gives that
\begin{equation*}
 H_T(t) \sim 2 \pi \left|T^2 - \frac{t^2}{4}\right|^{-1/2}  \left|\frac{t}{2}\right|^{-1} \exp\big(-\pi q(t,T)\big),
\end{equation*}
where $q(t,T) = |T + \frac{t}{2}| + |T - \frac{t}{2}| - 2T$ which is $0$ for $|t| \leq 2T$ and is $|t-2T|$ for $|t| > 2T$.  Then
\begin{equation*}
 I \sim \frac{8}{\pi} \int_0^{2T} \left(T^2 - \frac{t^2}{4}\right)^{-1/2} dt  + \frac{8}{\pi} \int_{2T}^{\infty} \left( \frac{t^2}{4} - T^2\right)^{-1/2} e^{-\pi(t - 2T)}dt  = 8 + O(T^{-1/2}). 
\end{equation*}
Thus we arrive at the conjecture $\| \phi \|_4^4 \sim \frac{3}{\pi} + \frac{6}{\pi} = \frac{9}{\pi}$ which after renormalization gives \eqref{L4Maass}.

\section{A mean value of central $L$-values}\label{Sec4}
This section is devoted to the proof of Theorem \ref{prop31}. 
 An inspection of the proof  indicates that improving the upper bound into an asymptotic formula with a power saving is, in a vague sense,  almost {\em equivalent} to a subconvexity bound for $L(1/2, \mathrm{sym}^2 f)$ for $k \rightarrow \infty$.  Possibly if one had such an asymptotic formula then one could instead use an amplifier and thus obtain subconvexity for $L(1/2, \mathrm{sym}^2 f \times g)$.\\

In the following we make constant use of $\varepsilon$-convention, i.e.\ the symbol $\varepsilon$ denotes an arbitrarily small positive constant whose value may change from occurrence to occurrence.  We start by expressing $L(1/2, \mathrm{sym}^2 f \times g)$ with $f \in B_k$, $g \in B_{2\kappa}$  by a standard approximate functional equation. 
The local factor at infinity is given by (combine \cite[Theorem 2]{Or} with \eqref{factorization})
\begin{displaymath}
 \Lambda_{k, \kappa}(s) :=  \begin{cases} (2\pi)^{-3s}\textstyle\Gamma(s + k+\kappa - \frac{3}{2})\Gamma(s+\kappa - \frac{1}{2})\Gamma(s+\kappa - k+\frac{1}{2}), & \kappa \geq k,\\
 (2\pi)^{-3s}\textstyle\Gamma(s + k+\kappa - \frac{3}{2})\Gamma(s+\kappa - \frac{1}{2})\Gamma(s+k-\kappa -\frac{1}{2}), & \kappa< k,
 \end{cases}
 \end{displaymath}
and the root number is $1$ if and only if one of the following two cases hold: $\kappa \geq k$ and $\kappa$ even, or  $\kappa < k$ and $\kappa$ odd. Otherwise the root number is $-1$. In the latter case Theorem \ref{prop31} is trivial, and we assume from now on that the root number is $+1$.  In this case we have
\begin{equation}\label{approx}
   L(1/2, \mathrm{sym}^2f  \times g) =2 \sum_{n, m} \frac{\lambda_g(n) A(m, n)}{n^{1/2} m} W(nm^2),
\end{equation}
where $W$ is a smooth weight function satisfying
\begin{equation}\label{W}
  x^j W^{(j)}(x) \ll_{j, A} \left(1 + \frac{x}{k^2}\right)^{-A}
\end{equation}
for any $j, A \geq 0$ if $\kappa = k + O(1)$. For instance, we can take
\begin{displaymath}
  W(x) = \frac{1}{2\pi i} \int_{(1)} \frac{\Lambda_{k, \kappa}(\frac{1}{2} + s)}{\Lambda_{k, \kappa}(\frac{1}{2})} \left(\cos\frac{\pi s}{10 A}\right)^{-60A} x^{-s} \frac{ds}{s}
\end{displaymath}
 (cf.\ e.g. \cite[Section 5.2]{IK}). With later applications in mind, we consider a slightly more general quantity
\begin{displaymath}
 \mathcal{M}_f(r ) :=  \frac{12}{2k-1} \sum_{g \in B_{2k}}  \lambda_g(r ) \frac{    L(1/2, \mathrm{sym}^2f  \times g)  }{  L(1, \mathrm{sym}^2 g)} 
 \end{displaymath}
 for an integer $0 < r < k^{1/10}$. By positivity and Deligne's bound we have
 \begin{equation}\label{est}
   \mathcal{M}_f( r)  \ll r^{\varepsilon} \mathcal{M}_f(1). 
 \end{equation}
  We can now apply the Petersson formula \eqref{peter} getting $\mathcal{M}_f(r ) = \mathcal{M}_f^{(1)}(r ) + \mathcal{M}_f^{(2)}( r)$ where $\mathcal{M}_f^{(1)}( r)$ is the diagonal term and $\mathcal{M}_f^{(2)}(r )$ is the off-diagonal contribution.  We have
 \begin{equation}\label{m1}
 \begin{split}
   \mathcal{M}^{(1)}_f(r )&= \frac{2}{\zeta(2)} \sum_{m} \frac{A(m, r)}{r^{1/2}m} W(m^2) \ll   k^{\varepsilon}   \end{split}  
 \end{equation}
 by Deligne's bound $A(m, r) \ll (rm)^{\varepsilon}$ (or Iwaniec's method \cite{IwaniecSpectralgrowth}) and \eqref{W}. We proceed to analyze the off-diagonal contribution
 \begin{equation}\label{m2}
   \mathcal{M}^{(2)}_f(r ) =  \frac{4\pi i^{k}}{\zeta(2)} \sum_{n, m, c} \frac{ A(m, n)}{n^{1/2} m} W(nm^2)  \frac{S(n, r, c)}{c} J_{2\kappa-1}\left(\frac{4\pi \sqrt{nr}}{c}\right).  
  \end{equation}
The multiple sum is absolutely convergent. 
 By \eqref{W} we can truncate the $n$-sum at $n\leq k^{2+\varepsilon} m^{-2}$. We insert smooth partitions of unity for the $n$ and $c$-sums, and are left with bounding 
\begin{displaymath}
  \mathcal{M}^{(2)}_f(r, N, C) =  \sum_{m, c} \frac{\Omega_1(c/C)}{mCN^{1/2}}    \left. \sum_{d \, (c )}\right.^{\ast}  e\left(\frac{dr}{c}\right)\sum_{n}  A(m, n)   e\left(\frac{\bar{d}n}{c}\right) \Omega_2\left(\frac{n}{N}\right)J_{2\kappa-1}\left(\frac{4\pi \sqrt{nr}}{c}\right)
\end{displaymath}
for
\begin{equation}\label{ranges}
   N  \leq \frac{k^{2+\varepsilon}}{m^2}, \quad C\leq 100 \frac{\sqrt{Nr}}{k}, 
\end{equation}
  the latter truncation coming from the decay properties of the Bessel function near 0. Here $\Omega_1$ and $\Omega_2$ are  fixed, smooth, compactly supported weight functions. We remark   that   \eqref{ranges} implies 
\begin{equation}\label{ranges1}
   k^2r^{-1} \leq N \leq k^{2+\varepsilon}, \quad   cm \ll r^{1/2} k^{\varepsilon}.
 \end{equation}  
We apply the Voronoi formula \eqref{Vor} with 
\begin{equation}\label{littlepsi}
   \psi(n)= \psi_{N, c, r}(n) =  \Omega_2\left(\frac{n}{N}\right) J_{2\kappa-1}\left(\frac{4\pi \sqrt{nr}}{c}\right). 
\end{equation}   
We define $\Psi^{\pm}$ as in \eqref{Psi}.  Then the Voronoi formula \eqref{Vor}   implies
\begin{equation}\label{aftervor}
  \mathcal{M}^{(2)}_f(r, N, C) = \sum_{m, c} \frac{\Omega_1(c/C)}{mCN^{1/2}}  c \sum_{n_1 \mid c} \sum_{\pm} \sum_{n_2} \frac{A(n_2, n_1)}{n_1n_2 }   \left. \sum_{d \, (c )}\right.^{\ast}  e\left(\frac{dr}{c}\right) S(md, \pm n_2, c/n_1) \Psi^{\pm}\left(\frac{n_2n_1^2}{c^3 m}\right).
\end{equation}
We need the following two technical lemmas.

\begin{lemma}\label{lem32} With $\psi$ as in \eqref{littlepsi} and under the assumption \eqref{ranges} we have
\begin{displaymath}
  \Psi^{\pm}(x) \ll_{A, \varepsilon}k^{\varepsilon}\left(\frac{x^{1/2} c}{r^{1/2}} +  \frac{ xc^2}{r}\right) \left(1+\frac{x}{Xk^{\varepsilon} }\right)^{-A}, \quad X := \frac{N^{1/2}r^{3/2}}{ c^{3}}\,\, \,\,(\gg k^{2+o(1)}). 
\end{displaymath}
\end{lemma}

In our situation $ x \geq 1/(c^3m)$, hence $xc^2/r \geq 1/(mcr)$. Hence \eqref{ranges1} implies  the slightly simpler bound
\begin{equation}\label{simpler}
  \Psi^{\pm}(x) \ll_{A, \varepsilon}k^{\varepsilon} \frac{xc^2}{r^{1/4}}  \left(1+\frac{x}{Xk^{\varepsilon} }\right)^{-A}, \qquad x \geq \frac{1}{c^3m}. 
  \end{equation}

\begin{lemma}\label{lem33} We have
\begin{displaymath}
  \Bigl|\left. \sum_{d \, (c )}\right.^{\ast}  e\left(\frac{dr}{c}\right) S(md, \pm n_2, c/n_1)\Bigr| \leq \tau(c )  c (c, m) 
 \end{displaymath}  
 where $\tau(c )$ denotes the number of divisors of $c$.
\end{lemma}

Coupling these results with Deligne's bound, it follows by straightforward estimates for \eqref{aftervor} and \eqref{ranges1} that
  $\mathcal{M}_f^{(2)}(1, N, C) \ll k^{\varepsilon}$. 
This concludes the proof of  Theorem \ref{prop31}. It remains to prove the two lemmas.\\

{\it Proof of Lemma \ref{lem32}.} By \cite[6.561.14]{GR} we have
\begin{displaymath}
  \tilde{\psi}(-s) = \frac{1}{2\pi i} \int_{(\nu)} \tilde{\Omega}_2(u) \left(\frac{2\pi\sqrt{r}}{c}\right)^{2s+2u} \frac{\Gamma(\kappa - s - u - \frac{1}{2})}{\Gamma(\kappa+s+u+\frac{1}{2})} N^u du
\end{displaymath}
where $\tilde{\Omega}_2$ denotes the Mellin transform of $\Omega_2$, which is an entire function that is rapidly decaying on vertical lines. Here and in the following we write $u = \nu+ iw$, and as usual  $s = \sigma + it$. We conclude
\begin{displaymath}
    \Psi^{\pm}(x) = \frac{1}{2\pi^{3/2}}  \int_{(\sigma)}  \int_{(\nu)}  \tilde{\Omega}_2(u) \frac{(2\sqrt{r}/c)^{2s+2u}}{\pi^{s-2u}}   G^{\pm}(s) \frac{\Gamma(\kappa - s - u - \frac{1}{2})}{\Gamma(\kappa+s+u+\frac{1}{2})} N^u  x^{-s}  \frac{ds}{2\pi i} \frac{du}{2\pi i}
\end{displaymath}
with $G^{\pm}$ as in \eqref{G}.   A simple version of Stirling's formula shows
\begin{equation}\label{b1}
  G^{\pm}(s) \ll_{\sigma} (k+|t|)^{2\sigma+1} (1+|t|)^{\sigma+\frac{1}{2}}
\end{equation}  
and 
\begin{equation}\label{b2}
   \frac{\Gamma(\kappa - s - u - \frac{1}{2})}{\Gamma(\kappa+s+u+\frac{1}{2})}  \ll_{\sigma, \nu} (k + |t +w|)^{-2\sigma-2\nu - 1}.
\end{equation}
for any fixed $ \sigma, \nu > -1$, and we also recall $\tilde{\Omega}_2(u) \ll_A (1+|w|)^{-A}$. In particular, the double integral is absolutely convergent for  $2\nu > \sigma + 3/2$.\\

We first show that $\Psi^{\pm}$ is rapidly decaying for  $x > X$. To this end we shift the two contours to $\Re s = A$ and $\Re u = A/2 + 3/4 + \varepsilon$ for some large $A$ and small $\varepsilon > 0$. By trivial bounds together with \eqref{b1} and \eqref{b2}, we obtain
\begin{displaymath}
  \Psi^{\pm}(x)  \ll_{\varepsilon, A}      \frac{(Nr)^{3/4}k^{\varepsilon}}{c^{3/2}}   \left(\frac{xc^3}{N^{1/2}r^{3/2}}\right)^{-A}. 
\end{displaymath}
Changing $A$ and $\varepsilon$ if necessary, this is sufficient in the range $x \geq Xk^{\varepsilon}$. \\

Next we investigate the range $x \leq Xk^{\varepsilon}$. Here we shift the $s$-contour to $\Re s = -1/2$. Shifting the $u$-contour to the far right, we see that we can truncate the $s$-integration at 
\begin{displaymath}
  |t| \leq T := \frac{N^{1/2} r^{1/2} k^{\varepsilon}}{c} = \frac{Xc^2k^{\varepsilon}}{r}
 \end{displaymath} 
at the cost of a negligible error. Having done the truncation (in a smooth fashion) we shift the contour back to $\Re u = 0$, and truncate the $u$-integration at $|w| \leq k^{\varepsilon}$ again at the cost of a negligible error. Hence we see that 
\begin{equation}\label{integral}
  \Psi^{\pm}(x) \ll \frac{k^{\varepsilon} x^{1/2}c}{r^{1/2}} \Biggl(\sup_{|w| \leq k^{\varepsilon}} \Bigl| \int_{-\infty}^{\infty} \omega(t) \left(\frac{4r}{\pi c^2 x}\right)^{it}    G^{\pm}\left(-\frac{1}{2} + it\right) \frac{\Gamma(\kappa -it-iw)}{\Gamma(\kappa+it+iw)}  dt\Bigr| + O(k^{-10})\Biggr)
\end{equation}
where $\omega$ is a smooth function with $\omega(t) = 1$ for $|t| \leq T$, $\omega(t) = 0$ for $|t| \geq 2 T$ and $\omega^{(j)}(t) \ll_j |t|^{-j}$ for all $j \in \Bbb{N}_0$. 

We need to show square-root cancellation in the $t$-integral which follows from the stationary phase method. The argument is greatly simplified by the following observation: by well-known properties of the Gamma-function we have 
\begin{displaymath}
  G^{\pm}(-1/2 + it) =  \mp i  \frac{2^{1/2 - 3 i t} \Gamma(1/2 + it) \exp(\pm \frac{i\pi}{4}(1+2it))\Gamma(k-1/2 + it)}{\sqrt{\pi}\Gamma(k-1/2 - it)}. 
\end{displaymath}
Hence the $t$-integral in  \eqref{integral} contains the term
\begin{displaymath} 
   H_k(t, w)  = \frac{\Gamma(k-1/2 + it)}{\Gamma(k - 1/2 - it)}  \frac{\Gamma(\kappa- it-iw)}{\Gamma(\kappa + it+iw)}
\end{displaymath}
which is almost constant (for small $w$). We see now the phenomenon mentioned in the introduction that large parts of the Voronoi kernel $G^{\pm}$ are almost cancelled by the Mellin transform of the Bessel-function from Petersson's formula, as long as $k \approx \kappa$.   We note that Stirling's formula implies
 \begin{equation}\label{g2}
  \mp i \frac{ 2^{\frac{1}{2}-3it}}{\sqrt{\pi}} \Gamma(1/2+it)\exp\Bigl(\pm \frac{i\pi}{4}(1+2it)\Bigr) =   \exp\Bigl(it \log \frac{|t|}{8e}\Bigr) v_{\pm}(t)  + O\big((1+|t|)^{-10}\big)
  \end{equation}  
for a smooth function $v_{\pm}$ satisfying $v_{\pm}^{(j)}(x) \ll   x^{-j}$ for all $j \in \Bbb{N}_0$. Putting it all together, the integral in \eqref{integral} equals
\begin{equation}\label{int1}
   \int_{-\infty}^{\infty} \omega(t)  v_{\pm}(t) H_k(t, w)   \exp\Bigl(it \log \frac{|t|r}{2\pi e c^2 x}\Bigr)   dt + O(1). 
\end{equation}
Since
\begin{displaymath}
  \frac{d^n}{dz^n} \frac{\Gamma'(z)}{\Gamma(z)}  \ll |z|^{-n}
\end{displaymath}
for $n \geq 1$, it is not hard to see that
\begin{displaymath}
  \frac{\partial^n}{\partial t^n} H_{k}(t, w) \ll \left(\frac{(1+|w|)}{|t|}\right)^{n} \ll  \left(\frac{k^{\varepsilon}}{|t|}\right)^{n}
\end{displaymath}
if $\kappa = k + O(1)$. 

Now we integrate trivially in \eqref{int1} for $|t| \leq k^{\varepsilon}$. There is one stationary point at $|t _0| = 2\pi x c^2/r$. We cut the remaining integral in $O(k^{\varepsilon})$ subintegrals over (smoothed) dyadic intervals of the form $[V_1,2 V_1]$ 
and assume without loss of generality that $t_0$ is the midpoint of one of the intervals. For all regions not containing $t_0$ we apply integration by parts in the form of Lemma \ref{integrationbyparts} below with $X = Y = 1$, $U  = V$, $Q = V_1$, $ R \asymp k^{-\varepsilon}$ to see that these are negligible. For the region containing $t_0$ we apply  Proposition \ref{statphase}  with $X=1$, $Y=Q = xc^2/r$ and  $V \asymp Q/k^{\varepsilon}$, so that altogether \eqref{int1} is at most $\ll k^{\varepsilon} + (xc^2/r)^{1/2}$.   This completes  the proof of Lemma \ref{lem32}.  
$\square$\\

{\it Proof of Lemma \ref{lem33}.} This is a straightforward computation. Interchanging sums, we find
\begin{displaymath}
\begin{split}
&  \left. \sum_{d \, (c )}\right.^{\ast}  e\left(\frac{dr}{c}\right) S(md, \pm n_2, c/n_1) = \underset{h \, (c/n_1)}{\left.\sum\right.^{\ast}} e\left(\frac{\pm n_2 \bar{h}}{c/n_1}\right) r_c(r + mhn_1) \\
  = & \sum_{f \mid c} f \mu\left(\frac{c}{f}\right)  \underset{\substack{h \, (c/n_1)\\ mhn_1 \equiv -r \, (f)}}{\left.\sum\right.^{\ast}}e\left(\frac{\pm n_2 \bar{h}}{c/n_1}\right), 
\end{split}  
\end{displaymath}
and this is trivially bounded by
\begin{displaymath}
  \sum_{f \mid c} f \cdot \frac{c}{n_1} \cdot \frac{(f, mn_1)}{f} \leq \tau(c ) c (c, m),
\end{displaymath}
as claimed. \hfill $\square$

 \section{Proof of Theorem \ref{thmSK}}

The proof of Theorem \ref{thmSK} uses heavily the analysis of  the preceding section. For odd $k$ we consider the quantity
\begin{equation}\label{quant}
   \mathcal{S}  := \frac{12}{2k-1} \sum_{g \in B_{2k}}  \frac{\pi^2}{15 \, L(3/2, g) L(1, \mathrm{sym}^2 g)}\cdot  \frac{12}{k} \sum_{f \in B_{k+1}} L\left(\frac{1}{2}, \mathrm{sym}^2 f \times g\right). 
\end{equation}
The crucial point is to sum over $g$ first and postpone the $f$-average to the last possible moment. This different order of summation is the key to improving the result of \cite{LY}.  We will apply Theorem \ref{prop31} several times with $k+1$ (which is even) in place of $k$. We recall \eqref{approx} and
\begin{displaymath}
  \frac{1}{L(3/2, g)} = \sum_{(r, s) = 1} \frac{\mu(r )\mu(s)^2 \lambda_g(r )}{r^{3/2} s^3}
\end{displaymath}
 and insert both expressions into \eqref{quant}. We use the Petersson formula \eqref{peter} for the $g$-sum and obtain 
 \begin{displaymath}
 \mathcal{S} =   \frac{\pi^2}{15} \cdot \frac{12}{k} \sum_{f \in B_{k+1}}  \frac{L(1, \mathrm{sym}^2f)}{L(1, \mathrm{sym}^2f)}\sum_{(r, s) = 1} \frac{\mu(r ) \mu(s)^2}{r^{3/2} s^3} \bigl(\mathcal{M}^{(1)}_f(r ) + \mathcal{M}^{(2)}_f( r)\bigr),
 \end{displaymath}
where $\mathcal{M}^{(1)}_f(r )$ and $\mathcal{M}^{(2)}_f(r )$ were defined in \eqref{m1} and \eqref{m2}. We have inserted a redundant fraction   in order to ease the application of the Petersson formula later. The Dirichlet series for $L(1, \mathrm{sym}^2 f)$ is not absolutely convergent, but for almost all $f$ we can represent this value by a short Dirichlet polynomial. More precisely, the following holds:
\begin{lemma}\label{lem51}
Given $\delta_1, \delta_2 > 0$, there is $\delta_3 > 0$ such that
\begin{equation}\label{l1}
  L(1, \mathrm{sym}^2 f) = \sum_{d_1,d_2} \frac{\lambda_f(d_1^2)}{d_1d_2^2} \exp\left(-\frac{d_1d_2^2}{k^{\delta_1}}\right) + O(k^{-\delta_3})
\end{equation}
for all but $O(k^{\delta_2})$ cusp forms $f \in B_{k+1}$.
\end{lemma}
\begin{proof}
This follows from the zero-density estimate \cite[Theorem 1]{LW}: given $0 < \eta < 1/100$, define $$\mathcal{R}(\eta) := \{s \in \Bbb{C} \mid \sigma \geq 1 - \eta, |t| \leq 100k^{\eta}\} \cup \{s \in \Bbb{C} \mid \sigma \geq 1\}$$ and $B^{+}_{k+1}(\eta) := \{ f \in B_{k+1} \mid L(s, \mathrm{sym}^2f) \not = 0 \text{ for } s \in \mathcal{R}(\eta)\}$. Then $\#(B_{k+1} \setminus B_{k+1}^{+}(\eta)) \ll k^{31\eta}$ by \cite[(1.11)]{LW}. For $f \in B_{k+1}^{+}(\eta)$ it follows by standard complex analysis (see e.g.\ \cite[Lemma 2]{Lu}) that  $L(s, \mathrm{sym}^2f) \ll k^{\varepsilon}$ for $s \in \mathcal{R}(\eta/2)$. Let $\mathcal{C}(\eta)$ denote the boundary of $\mathcal{R}(\eta/2)$. Then
\begin{displaymath}
   L(1, \mathrm{sym}^2 f) = \sum_{d_1,d_2} \frac{\lambda_f(d_1^2)}{d_1d_2^2} \exp\left(-\frac{d_1d_2^2}{k^{\delta_1}}\right)  - \int_{\mathcal{C}(\eta)} L(s, \mathrm{sym}^2f) \Gamma(s-1) k^{\delta_1(s-1)} ds
\end{displaymath}
for $f \in B_{k+1}^+(\eta)$, and the integral is $O(k^{-\delta_1 \eta/2 + \varepsilon})$. The lemma follows with $\delta_3 < \delta_1\delta_2/62$. 
 \end{proof}

By Lemma \ref{lem51} we obtain
 \begin{equation}\label{s1}
 \begin{split}
 \mathcal{S} &=   \frac{\pi^2}{15} \cdot \frac{12}{k} \sum_{f \in B_{k+1}}  \frac{1}{L(1, \mathrm{sym}^2f)}\sum_{d_1,d_2} \frac{\lambda_f(d_1^2)}{d_1d_2^2} \exp\left(-\frac{d_1d_2^2}{k^{\delta_1}}\right) \\
 &\times \sum_{(r, s) = 1} \frac{\mu(r ) \mu(s)^2}{r^{3/2} s^3} \bigl(\mathcal{M}^{(1)}_f(r ) + \mathcal{M}^{(2)}_f( r)\bigr) +O\bigl(k^{-\delta_3+\varepsilon} + k^{\delta_2 - 1+\varepsilon}\bigr).
\end{split} 
 \end{equation}
 The error term comes from two sources: the error in Lemma \ref{lem51} and the bad forms $f$ for which \eqref{l1} does not hold in which case we estimate trivially using \eqref{est} and Theorem \ref{prop31}. We proceed to estimate the two main terms in \eqref{s1} that we call $\mathcal{S}^{(1)}$ and $\mathcal{S}^{(2)}$.  By the Hecke relations we have
 \begin{displaymath}
\begin{split}
  \mathcal{S}^{(1)} = \frac{2\pi^2}{15\zeta(2)}  \cdot \frac{12}{k}& \sum_{f \in B_{k+1}} \frac{1}{L(1, \mathrm{\sym}^2 f)}    \sum_{\substack{d_1, d_2, a, m_1, m_2, r, s\\ (ar, s) = 1}} \sum_{h \mid (m_1^2, r^2)}  \frac{\mu(ar ) \mu(s)^2\mu(a) \lambda_f(m_1^2r^2/h^2) \lambda_f(d_1^2) }{r^{2} s^3a^3 m_1m_2^2d_1d_2^2}\\
  & \times W(rm_1^2m_2^4a^3)\exp\left(-\frac{d_1d_2^2}{k^{\delta_1}}\right). 
\end{split}  
\end{displaymath}
 We are now in a position to apply the Petersson formula a second time. The diagonal term equals
\begin{displaymath}
  \mathcal{S}^{(11)}  = \frac{2\pi^2  }{15 \zeta(2)^2} \sum_{\substack{d_2, a, m_1, m_2, r, s\\ (ra, s) = 1}} \sum_{h \mid (m_1^2, r^2)} \frac{\mu(a) \mu(s^2) \mu(ra) h}{r^3m_1^{2}a^3m_2^2s^3d_2^2} W(rm_1^2a^3m_2^4) \exp\left(-\frac{m_1rd_2^2/h}{k^{\delta_1}}\right).
\end{displaymath}
 By Mellin inversion and a straightforward computation with Euler products we obtain
\begin{displaymath}
   \mathcal{S}^{(11)} =  \frac{2\pi^2}{15 \zeta(2)^2}    \int_{(1)}\int_{(1)} L(u, v)\widetilde{W}(u) \Gamma(v) k^{\delta_1 v} \frac{du}{2\pi i} \frac{dv}{2\pi i}
\end{displaymath}
where
\begin{displaymath}
  L(u, v) :=  
\zeta(2 + 4u) \zeta(2 + 2u+v)\zeta(2+2v) 
\prod_p\left(1+\frac{1}{p^3} - \frac{1}{p^{3+u+v}} - \frac{1}{p^{4 + 3u+v}}\right).
\end{displaymath}
We shift   the contours to $\Re u =\Re v = -1/5$, pick up the poles of $\widetilde{W}$ and $\Gamma$ at $u = 0$ and $v=0$ and obtain
\begin{equation}\label{s11}
  \mathcal{S}^{(11)}  = \frac{2\pi^2 \zeta(2)}{15 \zeta(2)^2} \frac{\zeta(2)^2}{\zeta(4)} + O(k^{-2/5} + k^{-\delta_1/5}) = 2 + O(k^{-2/5} + k^{-\delta_1/5}). 
\end{equation}
The off-diagonal contribution equals
\begin{displaymath}
\begin{split}
  \mathcal{S}^{(12)} = 2\pi i^{-k}  \frac{2\pi^2}{15\zeta(2)^2}  &     \sum_{\substack{d_1, d_2, a, m_1, m_2, r, s\\ (ar, s) = 1}} \sum_{h \mid (m_1^2, r^2)} \sum_{c}  \frac{\mu(ar ) \mu(s)^2\mu(a)  }{c r^{2} s^3a^3 m_1m_2^2d_1d_2^2} S\left(\frac{m_1^2r^2}{h^2}, d_1^2, c\right)\\
  & \times W(rm_1^2m_2^4a^3)\exp\left(-\frac{d_1d_2^2}{k^{\delta_1}}\right) J_{k}\left(\frac{m_1rd_1}{hc}\right).
\end{split}  
\end{displaymath}
By the rapid decay of the Bessel function near 0 we can truncate the $c$-sum at $c \leq 100 \frac{m_1rd_1}{hk}$. We use  the trivial bounds 
\begin{equation}\label{trivialbound}
  |S(\ast, \ast, c)| \leq c, \quad  J_{k}(x) \ll k^{-1/3}
\end{equation}  
   to see that 
\begin{equation}\label{s12}
  \mathcal{S}^{(12)} \ll k^{-1/3+\delta_1 + \varepsilon}. 
\end{equation}

Next we turn to the estimation of $\mathcal{S}^{(2)}$. Let $0 < \delta_4 < 1/10$. By \eqref{est} and Theorem \ref{prop31} we can truncate the $r$-sum at $r \leq k^{\delta_4}$ at the cost of an error $O(k^{-\delta_4/2 +\varepsilon})$.  Hence we are left with bounding
\begin{displaymath}
\mathcal{S}^{(2)}(N, C) := \frac{12}{k} \sum_{f \in B_{k+1}}  \frac{1}{L(1, \mathrm{sym}^2f)}\sum_{d_1,d_2} \frac{\lambda_f(d_1^2)}{d_1d_2^2} \exp\left(-\frac{d_1d_2^2}{k^{\delta_1}}\right)  \sum_{\substack{r \leq k^{\delta_4}\\(r, s) = 1}} \frac{\mu(r ) \mu(s)^2}{r^{3/2} s^3}  \mathcal{M}^{(2)}_f( r, N, C)
\end{displaymath}
with $\mathcal{M}^{(2)}_f(r, N, C)$ as in \eqref{aftervor} and $N, C$ as in \eqref{ranges}. 
We insert Lemmas \ref{lem32} (in the form of \eqref{simpler}) and \ref{lem33} and conclude
\begin{displaymath}
\mathcal{S}^{(2)}(N, C) \ll \sum_{d_1 \leq k^{\delta_1+\varepsilon}} \sum_{r \leq k^{\delta_4}}\sum_m \sum_{C \leq c \leq 2C} \mathcal{T}(d_1, r, m, c, N) + O(k^{-100})
\end{displaymath}
where
\begin{displaymath}
\begin{split}
\mathcal{T}(d_1, & r, m, c, N) = \underset{\substack{ n_2 n_1^2 \leq k^{\varepsilon} N^{1/2} r^{3/2} m\\     n_1 \mid c}}{\sum } \frac{n_1 \tau(c ) (c, m)}{d_1 r^{7/4}m^2 N^{1/2}}   \Bigl | \frac{12}{k} \sum_{f \in B_{k+1}} \frac{\lambda_f(d_1^2) A(n_2, n_1)}{L(1, \mathrm{sym}^2 f)}\Bigr|\\
 &\ll k^{\varepsilon} \sum_{\substack{a, l_1, l_2, n_1, n_2\\ a^3l_1^2l_2n_1^2n_2 \leq k^{\varepsilon} N^{1/2} r^{3/2} m\\ al_1n_1 \mid c}}  \frac{al_1n_1 \tau(c ) (c, m)}{d_1 r^{7/4}m^2 N^{1/2}}   \sum_{h \mid (n_1^2, n_2^2)} \Bigl | \frac{12}{k} \sum_{f \in B_{k+1}} \frac{\lambda_f(d_1^2) \lambda_f(n_1^2n_2^2/h^2)}{L(1, \mathrm{sym}^2 f)}\Bigr|. 
 \end{split}
\end{displaymath} 
One last time we apply the Petersson formula. For the off-diagonal term we apply as before only the trivial bounds \eqref{trivialbound} and truncate the series appropriately by the rapid decay of the Bessel function near 0. Hence 
\begin{displaymath}
  \mathcal{T}(d_1,  r, m, c, N)  \ll \sum_{\substack{a, l_1, l_2, n_1, n_2\\ a^3l_1^2l_2n_1^2n_2 \leq k^{\varepsilon} N^{1/2} r^{3/2} m\\ al_1n_1 \mid c}}  \frac{al_1n_1 \tau(c) (c, m)}{d_1 r^{7/4}m^2 N^{1/2}}   \sum_{h \mid (n_1^2, n_2^2)} \left(\delta_{d_1h = n_1n_2} + O\left(\frac{ d_1n_1n_2}{hk^{4/3}}\right)\right). 
\end{displaymath}
Now it's just a matter of book-keeping, but we can simplify our task by noticing that \eqref{ranges} and \eqref{ranges1} imply that $m$ and $c$ and hence $a, l_1, n_1$ are $O(k^{\delta_4/2 +\varepsilon})$, and $h = O(k^{\delta_4+\varepsilon})$. Hence 
\begin{displaymath}
 \mathcal{S}^{(2)}(N, C) \ll k^{100(\delta_4+\delta_1) - 1} \Bigl(1 + \sum_{l_2n_2 \leq k^{1 + 100\delta_4}} \frac{n_2}{k^{4/3}}\Bigr) \ll k^{-1/3 + O(\delta_4 + \delta_1)}.
\end{displaymath}
Combining this with \eqref{s1}, \eqref{s11} and \eqref{s12} and choosing $\delta_1, \delta_2, \delta_4$ sufficiently small, the proof is complete.\\


\section{A geodesic restriction problem}\label{sec6}

In this section we prove Theorem \ref{cor2}.  For convenience of the reader, we first indicate a proof of \eqref{period}. By \eqref{normholo}, an $L^2$-normalized cuspidal Hecke eigenform has the Fourier expansion
\begin{equation}\label{fourfirst}
  f(z) =  a_f(1) \sum_{n=1}^{\infty} \lambda_f(n) (4\pi n)^{(k-1)/2} e(nz), \quad |a_f(1)|^2 = \frac{2\pi^2}{L(1, \mathrm{sym}^2 f) \Gamma(k)}. 
\end{equation}
We compute the Mellin transform of $f(iy)y^{k/2}$:
\begin{displaymath}
  \int_0^{\infty} f(iy) y^{k/2} y^s \frac{dy}{y} =  a_f(1) \frac{2^{k/2}}{\sqrt{4\pi}} L(1/2 + s, f) \frac{\Gamma(s + \frac{k}{2}) }{(2\pi)^s}.  
\end{displaymath}
By Parseval we obtain
\begin{displaymath}
  \mathcal{I} = \frac{1}{2\pi} |a_f(1)|^2 \int_{-\infty}^{\infty} \frac{2^k}{4\pi} |L(1/2 + it, f)|^2  |\Gamma(it + \textstyle\frac{k}{2})|^2 dt,
\end{displaymath}
 and \eqref{period} follows.\\
 
We proceed to prove Theorem \ref{cor2}. We can spectrally decompose $f^2$ into cusp forms of weight $2k$ getting  
\begin{equation}\label{decomp}
  \mathcal{I} = \sum_{g \in B_{2k}} \int_{0}^{\infty} \langle F^2, G\rangle g(iy) y^k \frac{dy}{y} = \sum_{g \in B_{2k}}   \langle F^2, G\rangle  a_g(1) \frac{2^k}{\sqrt{4\pi}} L(1/2, g) \Gamma(k)
\end{equation} 
where 
\begin{displaymath}
  |a_g(1)|^2 =  \frac{2\pi^2}{L(1, \mathrm{sym}^2 g) \Gamma(2k)}
\end{displaymath}
   is defined as in \eqref{fourfirst} and $G(z) = g(z) y^k$.    We insert \eqref{watson} with $f=h$ and use Cauchy-Schwarz together with the bound
\begin{displaymath}
  \frac{2^k \Gamma(k)}{\Gamma(2k)^{1/2}} \ll k^{-1/4}
\end{displaymath}
to  conclude (again by positivity) 
\begin{displaymath}
 \mathcal{I} \ll k^{-3/4 + \varepsilon} \Bigl(\sum_{g \in B_{2k}} L(1/2, \mathrm{sym}^2 f \times g)\Bigr)^{1/2} \Bigl(\sum_{g \in B_{2k}} L(1/2,   g)^3\Bigr)^{1/2}.  
\end{displaymath}
For both factors on the right-hand side we have best possible bounds; the former is given in Theorem \ref{prop31}, the latter in \cite[Theorem 3.1.1, p.\ 36]{Pe}. \\
 
\textbf{Remark:} 
We also observe 
 that \eqref{decomp} indicates 
\begin{displaymath}
   k^{-1/2} \sum_{g \in B_{2k}} \frac{\langle F^2, G \rangle L(1/2, g)}{L(1, \mathrm{sym}^2 g)} = k^{o(1)},
\end{displaymath}
where each term in the sum is (on Lindel\"of) of order $k^{-1/4 + o(1)}$. Hence there is some cancellation in this sum, but not square-root cancellation; in other words, the real number  $\langle F^2, G\rangle$ seems to have a slight tendency to be positive. In this context we remark that in the case of Maa{\ss} forms, Bir\'o \cite{Bi} has given an interesting formula for the triple product itself (not the square of its absolute value) in terms of a triple product over $1/2$-integral weight forms. 

\section{A general stationary phase lemma with smooth weights}

The main result of this section evaluates asymptotically fairly arbitrary smooth oscillating integrals.  As mentioned in the introduction, this result is more general than needed for the immediate purposes of the present paper.

We begin with a preparatory lemma which records conditions under which repeated integration by parts shows that an oscillatory integral is very small. This is similar in spirit to \cite[Lemma 6]{JM}. 
\begin{lemma} \label{integrationbyparts}
 Let $Y \geq 1$, $X, Q, U, R > 0$, 
and suppose that $w$ 
 is a smooth function with support on $[\alpha, \beta]$, satisfying
\begin{equation*}
w^{(j)}(t) \ll_j X U^{-j}.
\end{equation*}
Suppose $h$ 
  is a smooth function on $[\alpha, \beta]$ such that
\begin{equation}
 |h'(t)| \geq R
\end{equation}
for some $R > 0$, and
\begin{equation}\label{diffh0}
h^{(j)}(t) \ll_j Y Q^{-j}, \qquad \text{for } j=2, 3, \dots.
\end{equation}
Then the integral $I$ defined by
\begin{equation*}
I = \intR w(t) e^{i h(t)} dt
\end{equation*}
satisfies
\begin{equation}
\label{eq:Ipartsbound}
 I \ll_A (\beta - \alpha) X [(QR/\sqrt{Y})^{-A} + (RU)^{-A}].
\end{equation}
\end{lemma}

This should be interpreted as follows: the integral $I$ is negligible if $RU$ and $QRY^{-1/2}$ are both significantly bigger than 1. The variables $X, Y$ measure the size of $w$ and $h$, the variables $U, Q$ the ``flatness" of $w$ and $h$. In practice, $R$, $Y$ and $Q$ are often not  independent. A typical case is that \eqref{diffh0} holds for $j=1$ as well, and one has $Y/Q \asymp R$. Then $RU$ is big, if roughly speaking $e^{i h(t)}$ oscillates more than $w$, and $QRY^{-1/2} \asymp Y^{1/2}$ is also big as long as $e^{i h(t)}$ has some oscillation. These are natural conditions away from the stationary point.  A nice feature of Lemma \ref{integrationbyparts} is that it can quickly show that $I$ is extremely small even if $QR/\sqrt{Y}$ and $RU$ are tending to infinity rather slowly.

\begin{proof}
 Define the differential operator
\begin{displaymath}
\mathcal{D}(f)(t) :=-\frac{d}{dx} \left(\frac{f}{ih'}\right)(t)
\end{displaymath}
for a smooth function $f$ with compact support, so that
\begin{equation}\label{applyD}
  \int_{-\infty}^{\infty} f(t) e^{i h(t)} dt =  \int_{-\infty}^{\infty} \mathcal{D}^n(f)(t) e^{i h(t)} dt
\end{equation}
for any $n \in \Bbb{N}_0$.   It is easy to see by induction that 
\begin{equation}\label{diffD}
  \mathcal{D}^{n}(f)(t) = \sum_{\nu = n}^{2n}  \sum_{\mu = 0}^{\nu}  \frac{f^{(\mu)}(t)}{h'(t)^\nu}   \sum_{2\gamma_2 + \ldots + \nu\gamma_\nu = \nu - \mu} c_{\nu, \mu, \gamma_2, \ldots, \gamma_{\nu}}  h^{(2)}(t)^{\gamma_2} \cdots h^{(\nu)}(t)^{\gamma_{\nu}}
 \end{equation}
for certain absolute coefficients $c_{\nu, \mu, \gamma_2, \ldots, \gamma_{\nu}} \in \Bbb{C}$ and any $n \in \Bbb{N}_0$.   Then
\begin{equation}
 |I| \leq (\beta - \alpha) \| \mathcal{D}^n(w) \|_{\infty} \ll (\beta - \alpha) X \sum_{\nu = n}^{2n} R^{-\nu}  \sum_{\mu = 0}^{\nu} U^{-\mu} \frac{Y^{\frac{\nu - \mu}{2}}}{Q^{\nu - \mu}},
\end{equation}
which quickly leads to \eqref{eq:Ipartsbound}.
\end{proof}

\begin{proposition}\label{statphase}
Let $0 < \delta < 1/10$, $X, Y, V, V_1, Q > 0$, $Z := Q + X + Y + V_1+1$,  and assume that
\begin{equation}\label{importantconditions}
Y \geq Z^{3 \delta}, \quad V_1 \geq V \geq \frac{QZ^{ \frac{\delta}{2}} }{Y^{1/2}}.\end{equation} Suppose that $w$ is a smooth function on $\Bbb{R}$ with support on an  interval $J$ of length $V_1$, satisfying
\begin{equation*}
w^{(j)}(t) \ll_j X V^{-j}
\end{equation*}
for all $j \in \Bbb{N}_0$. Suppose $h$ is a smooth function on $J$ such that there exists a unique point $t_0 \in J$ such that $h'(t_0) = 0$, and furthermore
\begin{equation}\label{diffh}
h''(t) \gg Y Q^{-2}, \quad h^{(j)}(t) \ll_j Y Q^{-j}, \qquad \text{for } j=1,2, 3, \dots, t \in J. 
\end{equation}
Then the integral $I  $ defined by
\begin{equation*}
I = \intR w(t) e^{i h(t)} dt
\end{equation*}
has an asymptotic expansion of the form
\begin{equation}
\label{eq:statphase}
I  = \frac{ e^{ih(t_0)}}{\sqrt{h''(t_0)}} \sum_{n \leq 3 \delta^{-1} A} p_n(t_0)  + O_{A,\delta}(Z^{-A}), \quad p_n(t_0) = \frac{\sqrt{2\pi} e^{\pi i/4}}{n!} \Big(\frac{i}{2 h''(t_0)}\Big)^n  G^{(2n)}(t_0),
\end{equation}
where $A > 0$ is arbitrary, and
\begin{equation}\label{defG}
G(t) = w(t) e^{i H(t)}, \qquad H(t) = h(t) - h(t_0) - \half h''(t_0) (t-t_0)^2.
\end{equation}
Furthermore, each $p_n$ is a rational function in $h'', h''', \dots$, satisfying
\begin{equation}
\label{eq:pnderiv}
\frac{d^j}{d t_0^j} p_n(t_0) \ll_{j,n} X(V^{-j} + Q^{-j}) \big((V^2 Y/Q^2)^{-n} + Y^{-n/3}\big).
\end{equation}
\end{proposition}
 
 \bigskip
 
  The leading term $$\sqrt{2\pi} e^{\frac{\pi i}{4}}  \frac{ e^{ih(t_0)}}{\sqrt{h''(t_0)}} w(t_0) \ll \frac{QX}{Y^{1/2}}$$ in this asymptotic expansion is well-known and can be found in many sources but it can be difficult to find the full expansion in the literature.  It is desirable to have such an expansion even for a (slightly) oscillating weight function $w$ (cf.\ the end of the proof of Lemma \ref{lem32} for an example)  in which case $V$ is a bit smaller than $V_1$. Flexibility of the parameters $V$ and $V_1$ is also useful in situations  where one has several stationary points moving towards each other (in which case one splits the range of integration into sufficiently small subintervals). 
  
 The conditions \eqref{importantconditions} and the bound \eqref{eq:pnderiv} imply automatically that each term in the asymptotic expansion \eqref{eq:statphase} is smaller than the preceding term. Observe that the second condition in \eqref{importantconditions}    cannot be relaxed much because if $V_1 \ll Q^{1-\varepsilon}/\sqrt{Y}$ then the trivial bound is smaller than the main term in \eqref{eq:statphase}.  
  
   \begin{corollary}
\label{cor:shortintegral}
 Assume the conditions of Proposition \ref{statphase}.  There exists a function $w_0(t)$ supported on the interval $[-1, 1]$ such that with any  $T \asymp Z^{\varepsilon} (h''(t_0))^{-1/2}$, we have
\begin{equation}
\label{eq:shortintegral}
 \intR w(t) e^{ih(t)} dt = \intR w(t) w_0\Bigl(\frac{t-t_0}{T}\Bigr) e^{ih(t)} dt + O_{A,\varepsilon}(Z^{-A}).
\end{equation}
\end{corollary}
We will derive Corollary \ref{cor:shortintegral} in the course of the proof of Proposition \ref{statphase}.  The nice feature here is that the trivial bound applied to the right hand side of \eqref{eq:shortintegral} is only slightly worse than the main term in Proposition \ref{statphase}, but the form of the expression may be easier to handle for further manipulations.  For example, one may wish to study a multi-dimensional oscillatory integral by focusing on one variable at a time.  If one applies stationary phase in terms of one of the variables, then the stationary point $t_0$ may then depend implicitly on the other variables; this may make the further analysis more challenging.  The right hand side of \eqref{eq:shortintegral} has the pleasant feature that $t_0$ only appears in the argument of $w_0$ and not in $h$, whereas it occurs in both the phase of $h$ and in the weight function in \eqref{eq:statphase}.

\begin{proof} 
Let $U \leq V$ be a parameter satisfying $$\frac{Y U^2}{Q^2} \geq Z^{\delta}, \quad \frac{YU^3}{Q^3} \leq 1. $$   This is possible for $0 < \delta \leq 1/10$ by  \eqref{importantconditions}.  
Fix a smooth, compactly-supported function $w_0$ satisfying $w_0(x) = 1$ for $|x| < 1/2$, and consider
\begin{equation*}
I_0 = \intR w(t) \Bigl(1-w_0\Bigl(\frac{t-t_0}{U}\Bigr)\Bigr) e^{ih(t)} dt. 
\end{equation*}
 Notice that with $f(t) = w(t) \Bigl(1-w_0\Bigl(\frac{t-t_0}{U}\Bigr)\Bigr)$, one has
\begin{equation}\label{diffhf}
 f^{(j)} \ll_j X U^{-j} \,\,\, (j = 1, 2, \ldots), \qquad  |h'(t)|  \gg  |t - t_0| \min_{|\xi - t_0| \leq t} |h''(\xi)| \gg  \frac{UY}{Q^2}  \,\,\, (t \in \text{supp}(f)).
\end{equation}
Then we apply Lemma \ref{integrationbyparts} with $\beta-\alpha = V_1$,  $R \asymp UY/Q^2$, to obtain
\begin{equation}
\label{eq:I0bound}
I_0 \ll_{A, \delta} Z^{-A},
\end{equation}
where $A > 0$ is arbitrarily large, since $U^2 Y/Q^2 \geq Z^{\delta}$.   
Hence 
\begin{equation*}
I =  \intR w(t) w_0\Bigl(\frac{t-t_0}{U}\Bigr) e^{ih(t)} dt + O_{A, \delta}(Z^{-A}) =: I_1 + O_{A, \delta}(Z^{-A}),
\end{equation*}
say.   By choosing $U \asymp Z^{\varepsilon} h''(t_0)^{-1/2}$, we obtain Corollary \ref{cor:shortintegral}.

Writing a Taylor expansion for $h(t)$ around $t_0$, we have
\begin{equation*}
h(t) = h(t_0) +  \frac{h''(t_0)(t-t_0)^2}{2!} + H(t),
\end{equation*}
where
\begin{equation*}
H(t) = \frac{h'''(t_0) (t-t_0)^3}{3!} + \dots.
\end{equation*}
 Notice that 
 \begin{equation*}
H' \ll \frac{U^2 Y}{Q^3}, \quad H'' \ll \frac{U Y}{Q^3}, \quad  H^{(j)} = h^{(j)} \ll Y Q^{-j}, \quad \text{for } j \geq 3.
\end{equation*}
By \eqref{importantconditions} this implies $H^{(j)} \ll U^{-j}$ for $j=1, 2, \dots$.
With this notation we recast $I_1$ as  
\begin{equation*}
I_1  = e^{ih(t_0)} \intR g(t) e^{ih''(t_0) (t-t_0)^2/2} dt, \qquad g(t) = w(t) w_0\Bigl(\frac{t-t_0}{U}\Bigr) e^{iH(t)}.
\end{equation*}
Observe that $g^{(j)} \ll X U^{-j}$.

This integral can be evaluated in a number of ways and its asymptotic expansion is easily found.  
One simple way is to write, for a small parameter $\varepsilon$ to be chosen in a moment,
\begin{equation*}
g(t) = \intR \widehat{g}(y) e(ty) dy = \int_{|y| \leq U^{-1}Z^{\varepsilon}} \widehat{g}(y) e(ty) dy + O_{\varepsilon, A}(Z^{-A}), 
\end{equation*}
reverse the orders of integration, complete the square, and evaluate the Gaussian integral.  It becomes
\begin{equation*}
I_1 = \frac{\sqrt{2\pi} e^{\pi i/4} e^{ih(t_0)}}{\sqrt{h''(t_0)}}   \int_{|y| \leq U^{-1}Z^{\varepsilon}} \widehat{g}(y) \exp\left(2 \pi i y t_0 - i \frac{2 \pi^2 y^2}{h''(t_0)}\right) dy + O_{\varepsilon, A}(Z^{-A}).
\end{equation*}
Next we note that $y^2/h''(t_0) \ll Y^{-1} Q^2 U^{-2}Z^{ 2\varepsilon} \leq Z^{2 \varepsilon-\delta}$.   Now we choose $\varepsilon   = \delta/4$, so that the preceding quantity is $O(Z^{-\delta/2})$.  Hence by another Taylor development we obtain  
\begin{equation*}
I_1 = \frac{\sqrt{2\pi} e^{\pi i/4} e^{ih(t_0)}}{\sqrt{h''(t_0)}} \Bigl(\sum_{n \leq N} \frac{1}{n!} \Big(\frac{- 2 \pi^2 i}{h''(t_0)}\Big)^n  \int_{|y| \leq U^{-1}Z^{\varepsilon}} y^{2n} \widehat{g}(y) e^{2 \pi i y t_0} dy  + O_{\delta, N}(XZ^{-\frac{\delta N}{2}+\varepsilon})\Bigr)+ O_{  A }(Z^{-A})
\end{equation*}
for any integer $N$. We choose $N = \lfloor 3 A \delta^{-1} \rfloor$. 
Next we extend the integral to the whole real line without making a new error term, and use
\begin{equation*}
\intR y^{m} \widehat{g}(y) e(y t_0) dy = \leg{-i}{2 \pi}^m g^{(m)}(t_0),
\end{equation*}
which gives
\begin{equation*}
I_1 = \frac{ e^{ih(t_0)}}{\sqrt{h''(t_0)}} \sum_{n \leq 3\delta^{-1} A} \frac{\sqrt{2\pi} e^{\pi i/4}}{n!} \Big(\frac{i}{2 h''(t_0)}\Big)^n  g^{(2n)}(t_0) + O_{\delta, A}(Z^{-A}).
\end{equation*}
This is the desired asymptotic expansion, upon noting that $g^{(m)}(t_0) = G^{(m)}(t_0)$ with $G$ as in \eqref{defG}, 
since $w_0(\frac{t-t_0}{U})$ is identically $1$ in a neighborhood of $t_0$.

To finish the proof, we show that \eqref{eq:pnderiv} holds.  We recall the definition of $H$ in \eqref{defG} and notice that $H^{(j)}(t_0) = 0$ for $j=0,1,2$, and $H^{(j)}(t_0) = h^{(j)}(t_0)$ for $j \geq 3$.  Then we see that $G^{(2n)}(t_0)$ is a sum of (scalar multiples of) terms of the form
\begin{equation*}
w^{(\nu_0)}(t_0) H^{(\nu_1)}(t_0) \dots H^{(\nu_l)}(t_0),
\end{equation*}
where $\nu_0 + \dots + \nu_l = 2n$.  Hence we see that
\begin{equation*}
G^{(2n)}(t_0) \ll X(V^{-2n} + (Q^3/Y)^{-2n/3}),
\end{equation*}
the two extreme cases being $\nu_0 = 2n$, and $\nu_0 = 0$, $\nu_1 = \nu_2 = \dots = \nu_l = 3$.  Then each time we differentiate $G^{(2n)}(t_0)$ with respect to $t_0$ we save either a factor $Q$ or a $V$, and so
\begin{equation*}
\frac{d^j}{d t_0^j} G^{(2n)}(t_0) \ll X(V^{-j} + Q^{-j}) (V^{-2n} + (Q^3/Y)^{-2n/3}).
\end{equation*}
By the easily verifiable formula
\begin{equation*}
\frac{d^{j}}{dx^j} \frac{1}{F(x)} =  \binom{j+1}{j} \sum_{l=0}^{j} \frac{(-1)^l}{1+l} \binom{j}{l} \frac{ \frac{d^j}{dx^j} (F(x)^l)}{F(x)^{1+l}}.
\end{equation*}
we also have that
\begin{equation*}
\frac{d^j}{dt_0^j} \frac{1}{(h''(t_0))^n} \ll Q^{-j} (Q^2/Y)^{n},
\end{equation*}
and \eqref{eq:pnderiv} follows.  \end{proof}

 \end{document}